\documentclass[reqno]{amsart}
\usepackage[top=1in, bottom=1in, left=1in, right=1in]{geometry}

\usepackage{amsfonts}
\usepackage{amsmath}
\usepackage{comment}
\usepackage{amssymb}
\usepackage{bbm}
\usepackage{comment}
\usepackage{mathrsfs}
\numberwithin{equation}{section}
\usepackage{times}
\usepackage{siunitx}
\usepackage[usenames,dvipsnames]{color}
\usepackage{comment}
\usepackage{xcolor}
\usepackage{mathtools}
\usepackage{bm}
\usepackage{esvect}
\usepackage{hyperref}
\hypersetup{
    colorlinks = true,
linkcolor={black},
urlcolor={blue},
citecolor={blue},    
urlcolor = {blue},
citebordercolor = {0.33 .58 0.33},
 linkbordercolor = {0.99 .28 0.23},
 breaklinks=true}

\newcommand{\kommentar}[1]{}

\newcommand{\F}{\mathbb{F}}
\newcommand{\R}{\mathbb{R}}

\newcommand{\N}{\mathbb{N}}
\newcommand{\Z}{\mathbb{Z}}

\newtheorem{thm}{Theorem}[section]

\newtheorem{lem}[thm]{Lemma}
\newtheorem{rem}[thm]{Remark}

\newcommand{\dd}{\;\mathrm{d}}

\theoremstyle{remark}

\usepackage{geometry}
\DeclarePairedDelimiter\floor{\lfloor}{\rfloor}

\title{ Cannonball Polygons with Multiplicities}

\author{Anji Dong, Katerina Saettone, Kendra Song and Alexandru Zaharescu}

\address{
Anji Dong: Department of Mathematics,
University of Illinois Urbana-Champaign,
Altgeld Hall, 1409 W. Green Street,
Urbana, IL, 61801, USA}
\email{anjid2@illinois.edu}
\address{
Katerina Saettone: Department of Mathematics,
University of Illinois Urbana-Champaign,
Altgeld Hall, 1409 W. Green Street,
Urbana, IL, 61801, USA}
\email{kas18@illinois.edu}

\address{
Kendra Song: Department of Mathematics,
University of Illinois Urbana-Champaign,
Altgeld Hall, 1409 W. Green Street,
Urbana, IL, 61801, USA}
\email{kendras4@illinois.edu}
\address{
Alexandru Zaharescu: Department of Mathematics,
University of Illinois Urbana-Champaign,
Altgeld Hall, 1409 W. Green Street,
Urbana, IL, 61801, USA and ``Simion Stoilow" Institute of Mathematics of the Romanian Academy, 
P. O. Box 1-764, RO-014700 Bucharest, Romania}
\email{zaharesc@illinois.edu}  
\begin{document}
\nocite{*}
\setcounter{tocdepth}{1}
\begin{abstract}
    We generalize the Cannonball Problem by introducing integer-valued and non-increasing arithmetic functions $w$. We associate these functions $w$ with certain polygons, which we call cannonball polygons. Through this correspondence, we show that for any $\mathcal{Z}\in\N$, there exists a cannonball polygon with multiplicity 8 and largest side of length $\mathcal{Z}$. Moreover, for any multiplicity $s$ greater than 8, we provide an asymptotic formula for the number of distinct classes of cannonball polygons with multiplicity $s$. 
\end{abstract}
\maketitle
\section{Introduction}\label{sec: Introduction}
During an expedition between 1585 and 1586 to Roanoke Island, Sir Walter Raleigh asked Thomas Harriot, the scientific advisor for the voyage, how to efficiently stack cannonballs. Around 1587, Harriot provided a formula for the number of balls that can be stacked within a square pyramidal stack. As part of the larger class of figurate numbers, square pyramidal numbers are well known.  The study of these numbers traces back to Archimedes and Fibonacci. More details about this sequence can be found on the On-Line Encyclopedia of Integer Sequences, under the listing A000330 \cite{oeisA000330}. Let $S_n$ denote the $n$-th square pyramidal number, which is defined by
\begin{align}
   S_n &= \sum_{i = 1}^n i^2= \frac{2n^3+3n^2+n}{6}.\label{defn: square pyramidal number} 
\end{align}

The Cannonball Problem, officially posed by Lucas in \cite{NAM_1875_2_14__336_0}, asks which numbers are both square and square pyramidal. Lucas further proposed and conjectured that the solutions $(m,n)$ to $ S_{n} = m^{2}$ are only $(0,0), (1,1),$ and $(24, 70)$. 

In 1918, Watson confirmed Lucas' conjecture that besides 0 and 1, the only nontrivial solution is $S_{24}=70^2$. Watson's proof relies on methods involving elliptic curves. Since Watson's proof, others have been able to take elementary approaches to the problem. For instance, \mbox{in 1984}, Ma transformed the Cannonball Problem into a Pell-type equation and used Diophantine analysis (see \cite{ma1984}). Later, Anglin pulled from Ma's ideas to provide an elementary proof to Lucas' original question (see \cite{anglin}).  

Recently, Anderson, Woodall, and the fourth author introduced and studied arithmetic polygons in \cite{amy}, which are polygons with the vertex $O$ and integer side lengths that satisfy the following conditions:

\begin{enumerate}
\item\label{I:ArithmeticPropertyLengths} As one traverses the polygon starting and finishing at $O$, each side has length one greater than the length of the side preceding it.
        \item\label{I:ArithmeticPropertyIntersection} For each side of the polygon, there is a line perpendicular to the side that passes through both $O$ and one of the vertices at either end of that side.
        \item\label{I:ArithmeticPropertyDegenerate} There are no degenerate vertices. That is, there are no vertices with an angle of $0$ or $\pi$ radians.
    \end{enumerate}

Anderson, Woodall, and the fourth author proved that for any odd number $m\geqslant3$, there is at least one arithmetic polygon with $m$ sides. Surprisingly, they also showed that there are no arithmetic polygons with an even number of sides $m\in\{4,6,8,\cdots, 50\}$, but there exists an arithmetic polygon that has $52$ sides. It is also shown in \cite{amy} that if the sides of an arithmetic polygon are integers $a+1$ through $c$, then there is a $b$ such that $a + 1 < b < c$, which
guarantees a solution to the relation
\begin{align}\label{arithmetic polygon relation}
    (a+1)^2 + (a+2)^2 + \cdots + b^2 = (b+1)^2 + (b+2)^2 + \cdots + c^2.
\end{align}

For instance, the first occurring Pythagorean triple $(3,4,5)$ satisfies \eqref{arithmetic polygon relation} when $a = 2, b = 3,$ and $c = 5$. The resulting arithmetic polygon is a right triangle, and it is the only arithmetic polygon having three sides. For another example, the existence of an arithmetic polygon \mbox{with 52} sides is based on the identity
\[
12^2+13^2+\cdots+50^2=51^2+\cdots+63^2.
\]

Returning to \eqref{arithmetic polygon relation}, we notice that this relation is equivalent to 
\begin{align*}
    S_{b} - S_{a} = S_{c} - S_{b},
\end{align*}
where $S_{n}$ is given by \eqref{defn: square pyramidal number}. This relation is reminiscent of the Cannonball Problem, where instead of considering differences of square pyramidal numbers, only one square pyramidal number is compared to a square.





Our target is a modified version of these arithmetic polygons. Consider a polygon with integer side lengths and a vertex denoted by $O$ that satisfies the last two properties of an arithmetic polygon, and moreover satisfies the following:
\begin{itemize}
    \item[] Fix a positive integer $s$. As one traverses the polygon starting and finishing at \textit{$O$}, the first side has length 1. 
    The lengths of the consecutive sides are in non-decreasing order. The number of sides with length 1 is at most $s$, and for any for $k\in\N$, the number of sides with length $k$ is in non-increasing 
    order with respect to $k$, except for the final side which may have any integer length $\mathcal{Z}$. 
\end{itemize}

For example, we draw 8 sides of length 1, 4 sides of length 2, 0 sides of length 3, and the final side of length 4. Here the side lengths are in non-decreasing order, but the multiplicity of the larger side length is in non-increasing order except for the final side. 

We call such a polygon a \textit{cannonball polygon} of multiplicity $s$ for the following connection to the Cannonball Problem. Let $g$ be a sequence. The \textit{sup norm}, denoted by $||\cdot ||_\infty$, of $g$ is defined as
\[
||g||_\infty = \max \{g(i):i\in\N\}.
\]

Let $w(i)$ be a weight function on $\N$ that satisfies the following properties:
\begin{align*}
&\text{1. $w(i)\geqslant 0.$ }\\
&\text{2. There exists some $k\in \Z^+$ such that $w(i)=0$ for all $i\geqslant k$.}\\
&\text{3. $||w||_\infty \leqslant s$,} \text{ for some $s\in\N.$}\\
&\text{4. $w$ is non-increasing in the domain $[1,k].$}
\end{align*}

We call a weight function that satisfies the four properties above a \textit{cannonball weight function}. Using the Pythagorean Theorem, we immediately see that any cannonball polygon of multiplicity $s$ with final side of \mbox{length $\mathcal{Z}$} satisfies
\begin{align}
    \mathcal{Z}^2 = \sum_{i=1}^\infty w(i) i^2,\label{weight function for cannonball polygon}
    \end{align}
where $w(i)$ is a cannonball weight function. Conversely, by the Chainsaw process described \mbox{in \cite{amy}}, if \eqref{weight function for cannonball polygon} is fulfilled for some final side of length $\mathcal{Z}$ with some cannonball weight function $w$, then there always exists a corresponding cannonball polygon. We consider any two cannonball polygons to be in the same class if they are constructed from the same weight function. Then, there is a one-to-one correspondence between a class of cannonball polygons and a cannonball weight function. Note that each class can contain at most $2^n$ polygons, \mbox{where $n$} represents the number of sides the polygons have. 

As one can see, the original Cannonball Problem correlates to the special case of our problem with the restriction that $||w||_\infty = 1$ in \eqref{weight function for cannonball polygon}. In other words, a nontrivial cannonball polygon of multiplicity 1 can only be constructed with the final side of length $70$. The reader is referred to Figure 1.1(a) in \cite{amy} for an example construction of such a cannonball polygon. 

Our first objective is to ensure the existence of the construction of cannonball polygons with a final side length of any desired integer. In particular, we consider a case of small \mbox{multiplicity $s$} close to the original Cannonball Problem. The following theorem allows us to uniformly construct cannonball polygons of multiplicity 8 with a final side of any integer length.

\begin{thm}\label{generalized cannonball}
    For any positive integer $\mathcal{Z}$, there exists a cannonball polygon of multiplicity 8 with final side of length $\mathcal{Z}$. 
\end{thm}
\par

Besides the existence of a cannonball polygon with multiplicity 8 and a final side that has a length of an arbitrary integer $\mathcal{Z}$, for large enough $\mathcal{Z}$, we are also interested in an asymptotic formula for the number of classes of cannonball polygons with multiplicity $s \geqslant 9$, which has a final side of length $\mathcal{Z}$. We then have the following result.

\begin{thm}\label{Theorem: Representation}
     Let $s \in \N$, and let $\mathfrak{C}_s(\mathcal{Z})$ denote the number of classes of cannonball polygons with multiplicity $s$ and largest side of length $\mathcal{Z}$. Then for any integer $\mathcal{Z}\geqslant \exp{\left(\frac{1}{2} e^{94}\right)}$,
    \[
    \left|\mathfrak{C}_9(\mathcal{Z}) -\frac{4}{125}\bigg(\Gamma\bigg(\frac{4}{3}\bigg)\bigg)^{9}\mathfrak{S}_{9,S}(\mathcal{Z}^2) \mathcal{Z}^{4}\right|\leqslant 10^{14}\mathcal{Z}^{4-\frac{1}{54}},
    \]
    where $\Gamma$ is the Gamma function, and the arithmetic factor $\mathfrak{S}_{9,S}$, defined by \eqref{Arithmetic Factor for s=9} and \eqref{Definition of V(q,a)} below, satisfies the inequality
\begin{align*}
\mathfrak{S}_{9,S}(\mathcal{Z}^2)&\geqslant \exp{\left(- e^{92} \right)}.  
\end{align*}
    A more general version for any $s\geqslant 9$ is provided in Lemma \ref{lem: representations of polygons}.
\end{thm}
\begin{rem}
    The number $\exp{\left(\frac{1}{2} e^{94}\right)}$ in Theorem \ref{Theorem: Representation} is enormous, and comes from Lemma \ref{Singular Series Extension} when we \\\mbox{need $q\geq e^{e^{45}}$} to obtain a sharp bound on $|V(q)|$. In fact, we could greatly reduce this number, but the trade-off would be a looser upper bound on $|V(q)|$, resulting in a larger exponent in the error term in Theorem \ref{Theorem: Representation}. We decide to keep a smaller error term instead. Same reasoning applies to the lower bound $\exp{\left(- e^{92} \right)}$ of $\mathfrak{S}_{9,S}(\mathcal{Z}^2)$.
\end{rem}
In light of the one-to-one correspondence between cannonball polygons and cannonball weight functions, we will reduce the proofs of Theorems \ref{generalized cannonball} and \ref{Theorem: Representation} to two auxiliary lemmas on cannonball weight functions, which are Lemmas \ref{sum of eight square pyramidal number} and \ref{Lemma: Representation} stated below. To prove \mbox{Lemma \ref{sum of eight square pyramidal number}}, we begin with Linnik's method \cite{linnik1943representation} by writing sums of six terms, and then reducing the problem to counting lattice points on an algebraic curve with congruence constraints. To do so, we make use of Bombieri's theorem on exponential sums along algebraic curves, and a result by Cobeli and the fourth author \cite{cobeli} on counting points on algebraic curves satisfying congruence constraints. As we will see below, this leaves us with finitely many positive integers to check, which will be fulfilled via a simple and efficient algorithm. To prove \mbox{Lemma \ref{Lemma: Representation}}, we employ the Hardy-Littlewood Circle method. Through the use of analytic techniques and algebraic approaches, such as estimates of exponential sums and Hasse-Weil bound, we are able to obtain the asymptotic formula with an explicit power-saving error term.    

\subsection*{Structure of the Paper} 
The paper is organized as follows. In Section \ref{sec: general notations}, we introduce some standard notation used in the paper. In Section \ref{sec: auxiliary lemmas}, we prove the two auxiliary lemmas mentioned above and show that our two main theorems can be proved assuming these two lemmas. \mbox{Section \ref{sec: Proof of Lemma 3.1}} is devoted to the proof of the first auxiliary lemma, and thus finishes the proof of Theorem \ref{generalized cannonball}. In Sections \ref{sec: initial setup} to \ref{sec: Proof of Main Theorem}, we prove the second auxiliary lemma using the circle method, and thus finish the proof of Theorem \ref{Theorem: Representation}.
\section{General Notations}\label{sec: general notations} 
 We employ some standard notation that will be used throughout the article.

\begin{itemize}
    \item Throughout the paper, the expressions $f(X)=O(g(X))$, $f(X) \ll g(X)$, and $g(X) \gg f(X)$ are equivalent to the statement that $|f(X)| \leqslant C|g(X)|$ for all sufficiently large $X$, where $C>0$ is an absolute constant. 
    \item We define by $d(n)$, the divisor function defined on $\mathbb{N}$ by $d(n) = \sum_{d \mid n}1.$
    \item We write $e(\theta)$ to denote the expression $e^{2\pi i \theta}$.
    \item Given $\alpha \in \mathbb{R}$, the notation $\|\alpha\|$ denotes the smallest distance of $\alpha$ to an integer.
    \item Given a prime $p$, the notation $\mathbb{F}_{p}$ refers to the set of residue classes modulo $p$.
\end{itemize}

\section{Auxiliary lemmas for Theorem \ref{generalized cannonball} and Theorem \ref{Theorem: Representation}}\label{sec: auxiliary lemmas}

In this section, we present the proofs of the two main theorems assuming two auxiliary lemmas. Thus, the problem is reduced to handling these two lemmas. We first show that the following lemma implies Theorem \ref{generalized cannonball}.

\begin{lem}\label{sum of eight square pyramidal number}
    Any positive integer $m$ can be written as a sum of at most eight square pyramidal numbers. 
\end{lem}
\begin{proof}[Proof of Theorem \ref{generalized cannonball}]
    Assuming Lemma \ref{sum of eight square pyramidal number}, let $m$ be an arbitrary positive integer. Since $m$ can be written as a sum of at most eight square pyramidal numbers, there exists $j\in\Z^+$ \mbox{and $n_1< n_2<\cdots< n_j\in\Z^+$} such \mbox{that $m = \sum_{i=1}^j h(i)f(n_i)$,} where $h(i)$ is a weight function and $\sum_{i=1}^j h(i)\leqslant 8$. 
    Then, define the function $w(i)$ such that 
    \begin{align*}
        &w(1)=w(2)=\cdots=w(n_1)=\sum_{i=1}^j h(i)\\
        &w(n_1+1)=w(n_1+2)=\cdots=w(n_2)=\sum_{i=2}^j h(i)\\
        &\cdots\cdots\cdots\\
        & w(n_{j-1}+1)=\cdots=w(n_j)=h(n_j).
    \end{align*}
    Thus $w(i)$ satisfies all four required properties of a weight function, and moreover, 
    \[
    m = \sum_{i=1}^\infty w(i) i^2.
    \]
Taking $m = \mathcal{Z}^2$ gives us Theorem \ref{generalized cannonball}.
\end{proof}
Using the same arguments as in Lemma \ref{sum of eight square pyramidal number}, we see that for any integer $s\geqslant 9$ and $\mathcal{Z}\in\N$, the number of classes of cannonball polygons of multiplicity $s$ and with final side of \mbox{length $\mathcal{Z}$} is equal to the number of representations of $\mathcal{Z}^2$ as a sum of $s$ square pyramidal numbers. Therefore, Theorem \ref{Theorem: Representation} can be deduced from the following lemma:
\begin{lem}\label{Lemma: Representation}
For any integer $m$, let $\mathcal{C}_s(m)$ denote the number of representations of $m$ as a sum of $s$ square pyramidal numbers. Then for all $m>e^{e^{94}}$, 
\begin{align*}
\bigg|\mathcal{C}_9(m) -\frac{4}{125}\bigg(\Gamma\bigg(\frac{4}{3}\bigg)\bigg)^{9}\mathfrak{S}_{9,S}(m) m^{2}\bigg|\leqslant 10^{14}m^{2-\frac{1}{108}},
\end{align*}
where $\Gamma$ is the Gamma function, and the arithmetic factor $\mathfrak{S}_{9,S}(m)$, defined by \eqref{Arithmetic Factor for s=9} and \eqref{Definition of V(q,a)} below, satisfies the inequality
\begin{align*}
\mathfrak{S}_{9,S}(m)&\geqslant \frac{1}{e^{ e^{92}}}.  
\end{align*}
A more general version for $s\geqslant 9$ is given in Lemma \ref{main Theorem for formula}.
\end{lem} 
\begin{proof}[Proof of Theorem \ref{Theorem: Representation}]
    Taking $m=\mathcal{Z}^2$ will immediately give us the result.
\end{proof}
Now the remaining work is to prove these two lemmas. 

\section{Proof of Lemma \ref{sum of eight square pyramidal number}} \label{sec: Proof of Lemma 3.1}

 In this section, we will prove Lemma \ref{sum of eight square pyramidal number}.  We first consider the sum of six square pyramidal numbers. 
Let
\[
f(x) = \frac{2x^3+3x^2+x}{6},
\]
and write
 \begin{align}
     \mathcal{S}_f(6) &= f(x+a+1)+f(x-a)+f(x+b+1)+f(x-b)+f(x+c+1)+f(x-c)\notag\\
     &= (2x^3+6x^2+7x+3)+(4x+4)\left(\frac{a^2+a}{2}+\frac{b^2+b}{2}+\frac{c^2+c}{2}\right),\label{sum of six pyramidal numbers}
 \end{align}
where $a,b,c<x$. 
Using analogous reasoning to Section 9.1 in \cite{ours} up to constants, by Gauss's Eureka Theorem, we can reduce the original statement to the following: for any sufficiently large positive integer $m$, there exist $x,\ell,r\in\Z^+$ such that
\begin{align*}
    &1.\quad 0<L-f(\ell)-f(r)< 6x^3\\
    &2.\quad (4x+4)\mid (L-f(\ell)-f(r)),
\end{align*}
where $L = m - (2x^3+6x^2+7x+3)$. We will refer to the above conditions as Condition 1 and Condition 2, respectively. 
\subsection{The Arithmetic Conditions.} In this subsection, our goal is to reduce our current conditions to counting lattice points on an algebraic curve.  Fix $t,\varepsilon\in \R^+$. We aim to find a prime $p$ such that 
\begin{align}
    (t-\varepsilon)p^3\leqslant m\leqslant (t+\varepsilon)p^3.\label{p cubed range}
\end{align}

If the existence of such $p$ is guaranteed, then let $x = p-1$, so $4x+4 = 4p$. Now we find suitable $\ell$ and $r$ satisfying Conditions 1 and 2. Substituting $x=p-1$ into \eqref{p cubed range}, 
\begin{align}\label{p cubed range 2}
(t-\varepsilon)p^3\leqslant L+2(p-1)^3+6(p-1)^2+7(p-1)+2\leqslant (t+\varepsilon)p^3.
\end{align}

Assume that $p$ is large enough. We may reduce \eqref{p cubed range 2} to
\begin{align}
    (t-\varepsilon-2)p^3<L< (t+\varepsilon-2)p^3.
\end{align}

We first consider Condition 1, which will be satisfied if 
\begin{align}
    L-6p^3 <f(\ell)+f(r)<L. \label{Alternative form of condition 1}
\end{align}

Recalling the range of $L$, we need to choose $\ell$ and $r$ such that 
\begin{align}
    (t-\varepsilon-8)p^3<f(\ell)+f(r)<(t+\varepsilon-2)p^3.\label{f(l)+f(r) bound 2}
\end{align}

If $\ell,r\in((\alpha-\delta)p,(\alpha+\delta)p)$ for some $\alpha>0$, then 
\begin{align}
     f(\ell) + f(r) < \frac{2}{3}(\alpha + \delta)^{3}p^{3}.  \label{f(l)+f(r) upper bound}
\end{align}

Combining this with \eqref{f(l)+f(r) bound 2}, we may assume 
\begin{align*}
(\alpha + \delta)^{3} < \frac{3t}{2}  + \frac{3\varepsilon}{2}-3,
\end{align*}
which yields
\begin{align}\label{Alpha Range Lower}
\alpha < \bigg(\frac{3t}{2}  + \frac{3\varepsilon}{2}-3\bigg)^{1/3}-\delta.
\end{align}

Similarly, we need 
\begin{align}\label{flfr lower bound}
    \frac{2}{3}(\alpha - \delta)^{3}p^{3}<f(\ell) + f(r),
\end{align}
giving the lower bound
\begin{align}\label{Alpha Range Upper}
\alpha > \bigg(\frac{3t}{2}  - \frac{3\varepsilon}{2}-12\bigg)^{1/3}+\delta.
\end{align}

Now we turn to Condition 2. To this end, we wish to find an $\ell$ and $r$ that satisfy the congruence
\[
f(\ell) + f(r) \equiv L \bmod 4p.
\]
Since we assume earlier that $p$ is large enough, in particular, we may assume $p>2$. Therefore, $(p,4)=1$. Using the Chinese Remainder Theorem, it suffices to consider the congruence equations  
\begin{align*}
f(\ell)+f(r) &\equiv L \bmod 4 \quad \textrm{and} \quad f(\ell) + f(r) \equiv L \bmod p.
\end{align*}

The first equivalence condition is the same as
\begin{align}
    -\ell^3-r^3+\frac{1}{2}(\ell^2-\ell)+\frac{1}{2}(r^2-r)\equiv L \bmod 4.\label{first congruence}
\end{align}

Consider $\ell_0,r_0\in \Z/4\Z$ such that $\ell\equiv \ell_0\bmod 4$ and $r\equiv r_0\bmod 4$. We divide \eqref{first congruence} into four cases: \mbox{if $L\equiv 0\bmod 4$,} then let $\ell_0 =r_0 = 0$; if $L\equiv 1\bmod 4$, then let $\ell_0 = 0$ and $r_0=2$; if $L\equiv 2\bmod 4$, then we pick $\ell = r_0=2$; finally, if $L\equiv 3\bmod 4$, then we pick $\ell_0=0$ \mbox{and $r_0=1$.} In each case, the chosen $\ell_0$ and $r_0$ will satisfy \eqref{first congruence}. Therefore, our initial problem is reduced to finding $\ell,r\in\mathbb{Z^+}$ with the conditions that 
\begin{align}
r,\ell &\in \left( (\alpha - \varepsilon) p,(\alpha + \varepsilon) p\right),\notag\\
\ell &\equiv \ell_0 \bmod 4,\notag \\
r&\equiv r_0\bmod 4,\notag \\
\textrm{and} \quad  f(\ell)+f(r)&\equiv L\bmod p,\label{final conditions}
\end{align}
where $\ell_0$ and $r_0$ are determined by the congruence of $L\bmod 4$, and $\alpha$ satisfies \eqref{Alpha Range Lower} and \eqref{Alpha Range Upper}. \par
\subsection{Counting Lattice Points along Algebraic Curves}
Consider the algebraic curve
\begin{align}\label{f(x) formulas mod p}
f(x)+f(y)+B \equiv 0\bmod p,
\end{align}
where $B\in \Z$ and $p \geqslant 5$ is a prime. 
Since $(3,p)=1$, \eqref{f(x) formulas mod p} can be rewritten as
\begin{align}
        x^3+\frac{3}{2}x^2+\frac{1}{2}x+y^3+\frac{3}{2}y^2+\frac{1}{2}y+3B\equiv 0\bmod p.\label{reduced form}
\end{align}

The next lemma is a variation of \cite[Lemma~6.6]{ours}, which gives a bound for the number of solutions to our curve in \eqref{final conditions}.
\begin{lem}\label{elliptic curve satisfied}
    Suppose $p \geqslant 5$ and let $\mathcal{N}_{p}$ denote the number of solutions to \eqref{reduced form} over $\mathbb{F}_{p}$. Then 
\[-2\sqrt{p}\leqslant\mathcal{N}_{p}-(p+1)\leqslant p.
\]
\end{lem}
\begin{proof}

If we let $R_1=x+y+1$, $R_2=x-y$ and $B_0=12B-217$, then we may rewrite \eqref{reduced form} as 
\begin{align}
    R_1^3+3R_1R_2^2-R_1+B_0\equiv 0\bmod p.\label{R1 R2 conversion}
\end{align}

We consider cases when $B\equiv \frac{217}{12}\bmod p$, and  $B\not\equiv \frac{217}{12}\bmod p$. When $B\equiv \frac{217}{12}\bmod p$, \eqref{R1 R2 conversion} is reduced to
\begin{align}
    R_1(R_1^2+3R_2^2-1)\equiv 0\bmod p.\label{B0=0}
\end{align}
All the solutions to \eqref{B0=0} come from either $R_1\equiv 0\bmod p$ or $ R_1^2+3R_2^2-1\equiv 0\bmod p$. If both congruences are satisfied, we have 
\[
3R_2^2-1\equiv 0 \bmod p,
\]
which has at most 2 solutions. Note that $h(R_1, R_2) = R_1^2+3R_2^2-1$ is an irreducible non-singular curve over $\mathbb{F}_p$, so the number of solutions to $h(R_1, R_2)\equiv 0\bmod p$ is $p+1$. Thus, there are at most $2p+1$ and at least $2p-1$ solutions to \eqref{B0=0}. 

When $B\not\equiv \frac{217}{12}\bmod p$, we have $B_0\not\equiv 0\bmod p$, forcing $R_1\not\equiv 0\bmod p$. Therefore, multiplying \eqref{R1 R2 conversion} by $R_1^{-3}$, and writing $S_1=R_2R_1^{-1}, S_2=R_1^{-1}$, we get
\begin{align}
1+3S_1^2-S_2^2+B_0S_2^3&\equiv 0 \bmod p \label{S1 S2 conversion}.
\end{align}
We then multiply \eqref{S1 S2 conversion} by $27B^{2}_{0}$ and let $T_{1} = 9B_{0}S_{1}$ and $T_{2} = -3B_{0}S_{2}$ to obtain
\begin{align}\label{rewriting the congruence with T}
    T^{2}_{1} \equiv T^{3}_{2} + 3 T^{2}_{2} - 27B^{2}_{0} \bmod{p}.
\end{align}
Finally, we let $Z_{1} = T_{1}, Z_{2} = T_{2} + 1.$ Then \eqref{rewriting the congruence with T} becomes
\begin{align}\label{rewriting the congruence with Z}
    Z^{2}_{1} \equiv Z^{3}_{2} - 3 Z_{2} + \left(2 - 27B^{2}_{0} \right) \bmod{p},
\end{align}
which is an elliptic curve when
\begin{align}\label{the discriminant}
    -108+ 27 \left(2 - 27B^{2}_{0} \right)^{2} \not\equiv 0 \bmod{p}.
\end{align}
Note that \eqref{the discriminant} is not satisfied if and only if 
\begin{align}\label{j(Z_2)}
    j(Z_{2}) := Z^{3}_{2} - 3 Z_{2} + \left(2 - 27B^{2}_{0} \right) \bmod{p}
\end{align}
has at least two equal roots in the algebraic closure $\overline{\mathbb{F}}_{p}$ of $\mathbb{F}_{p}$. Suppose all roots to \eqref{j(Z_2)} are equal, denote by $\alpha$. Then,
\begin{align}\label{roots alpha}
    j(Z_{2})= (Z_{2} - \alpha)^{3}.
\end{align}
Equating the coefficients \eqref{roots alpha} and the right-hand side of \eqref{rewriting the congruence with Z}, we have $-3 \equiv 0 \bmod{p}$, which is a contradiction since $p \geqslant 5$ . 

Now suppose $j(Z_{2})$ has two roots, say $\alpha, \beta$ in the algebraic closure $\overline{\mathbb{F}}_{p}$ with $\alpha \neq \beta$ and $j(Z_{2}) = (Z_{2} - \alpha)^{2}(Z_{2} - \beta)$. Again, equating the coefficients, we obtain the system of equations
\begin{align}
    j(Z_{2})& \equiv (Z_{2} - \alpha)^{2}(Z_{2} + 2\alpha) \bmod{p}, \notag \\
    - 3& \equiv -3 \alpha^{2} \bmod{p} \notag, \\
    \text{ and }\quad 2 - 27B^{2}_{0}& \equiv 2 \alpha^{3} \bmod{p}. \label{alpha cubed}
\end{align}
Since $\alpha^2$ and  $\alpha^3$ are both invertible in $\F_p$, we must have $\alpha\equiv \pm 1\bmod p$. Substituting this into the third equation \mbox{of \eqref{alpha cubed}} forces either $B_0\equiv 0\bmod p$, which is a contradiction to our assumption, or $B_0\equiv\pm\frac{2}{27}\bmod p$. Therefore, if $B_0\not\equiv\pm \frac{2}{27}\bmod p$, the double root case does not exist, and so equation \eqref{rewriting the congruence with Z} is an elliptic curve. Otherwise, \eqref{rewriting the congruence with Z} is equivalent to
\begin{align}
Z_1^2 \equiv (Z_2-\alpha)^2(Z_2+2\alpha)\bmod p.\label{6.21}
\end{align} The number of solutions to \eqref{6.21} comes from either $Z_{2} \equiv \alpha \bmod{p}$ or $Z_{2} \not\equiv \alpha \bmod{p}.$ In the first case, when $Z_{2} \equiv \alpha \bmod{p},$ it forces $Z_{1} \equiv 0 \bmod{p}$, which implies $x = y$. Substituting $x = y$ in all variables, since $Z_{2} = T_{2} + 1$, we get
\begin{align}\label{getting what alpha equals}
    \alpha = T_{2} + 1 = -3B_{0} S_{2} + 1= \frac{-3B_{0}}{2x + 1} +1. 
\end{align}
\
Note that since we have supposed that $R_{1} \not\equiv 0 \bmod{p}, 2x + 1 \not\equiv 0 \bmod{p}.$ Therefore, \eqref{getting what alpha equals} is equivalent  to 
\begin{align*}
    x = \frac{1}{2} \left( \frac{-3B_{0}}{\alpha - 1} - 1 \right).
\end{align*}
This means that $x = y$ has a unique solution in $\mathbb{F}_{p}$, so the number of solutions in $\mathbb{F}_p$ to the first case is exactly $1$. In the latter case, when $Z_{2} \not\equiv \alpha \bmod{p},$ let $t = Z_{1}(Z_{2} - \alpha)^{-1}.$ Then \eqref{6.21} is equivalent to
\begin{align*}
    t^{2} \equiv Z_{2} + 2\alpha \bmod{p}.
\end{align*}

This is a nonsingular parabola, so the number of solutions is $p+1$. Summing the number of solutions of both cases, we conclude that when \eqref{rewriting the congruence with Z} is not an elliptic curve, \mbox{$\mathcal{N}_p = p+2$}. When equation \eqref{rewriting the congruence with Z} is an elliptic curve, we apply the Hasse-Weil Theorem, established \mbox{in \cite{hassecongruentfunctions},  \cite{hassecomplexmultiplication}, \cite{hassefiniteorder}, and \cite{weil1949},} and conclude that $\lvert\mathcal{N}_{p}-(p+1)\rvert\leqslant 2\sqrt{p}$. 

Combining all possible cases, we see that $\mathcal{N}_{p}$ satisfies 
\[
-2\sqrt{p}\leqslant\mathcal{N}_{p}-(p+1)\leqslant p.
\]

\end{proof}

Now we find Lehmer points satisfying the conditions in \eqref{final conditions} through an explicit version of a result by Cobeli and the fourth author \cite[Theorem 1]{cobeli}. We begin with notation \mbox{following \cite{cobeli}}. Let $1 \leqslant L_{0} \leqslant 4$ be such that $L \equiv L_{0} \bmod 4$, and  
\begin{align}\label{Curve definition}
\mathcal{C}: f(x_1)+f(x_2)-L=0
\end{align}
be an algebraic curve over $\mathbb{F}_{p}$.  Define the vectors
\begin{align}
\textbf{x}& := \begin{bmatrix}
           x_1 \\
           x_2 \\
         \end{bmatrix},
    \textbf{a} := \begin{bmatrix}
           4 \\
           4 \\
         \end{bmatrix},
    \textbf{b} := \begin{bmatrix}
           \ell_0 \\
           r_0 \\
         \end{bmatrix},
    \textbf{t}:= \begin{bmatrix}
         t_1 \\
          t_2 \\
         \end{bmatrix},
         \textbf{g}:= \begin{bmatrix}
         g_1 \\
          g_2 \\
         \end{bmatrix}. \label{vectors}
\end{align}

The set of Lehmer points with respect to prime $p$ and vectors $\textbf{a}$ and $\textbf{b}$ is defined as 
\begin{align}\label{Lehmer Point Definition}
\mathcal{L}(p,\mathcal{C},\textbf{a},\textbf{b}) := \{\textbf{x} : x_1 \equiv \ell_0\bmod 4, x_2\equiv r_0 \bmod 4, (x_1,x_2)\in\mathcal{C}, 0\leqslant x_1, x_2<p\} . 
\end{align}
 
 Consider the intervals $U_j = ((t_j-g_j)p, (t_j+g_j)p)$ for prime $p$, the vectors $\textbf{t}$ and $\textbf{g}$, and $j= 1,2$. Then the distribution function of the corresponding Lehmer points is
\begin{align}\label{Lehmer Distribution function}
F(p,\mathcal{C},\textbf{a},\textbf{b};\textbf{g},\textbf{t}):=\#\{\textbf{x}\in \mathcal{L}(p,\mathcal{C},\textbf{a},\textbf{b}):  x_j\in U_j, 0\leqslant t_j\pm g_j\leqslant 1, j=1,2 \}.
\end{align}

According to the above setup, we have the following lemma.
\begin{lem} \label{Z-C theorem}
Let $p$ be a prime and let $\mathcal{C}$, $\textnormal{\textbf{a}, \textbf{b}, \textbf{g}, \textbf{t}}$ and $ F$ be as defined in \eqref{Curve definition}, \eqref{vectors}, \eqref{Lehmer Point Definition} and \eqref{Lehmer Distribution function}. Then for $p\geqslant 10^{10}$,
\begin{align*}
F(p,\mathcal{C},\textnormal{\textbf{a}, \textbf{b}; \textbf{g}, \textbf{t}}) = \frac{g_1g_2}{4}(p+1)+ \mathcal{E},
\end{align*}
where $|\mathcal{E}| \leqslant 6.0009\sqrt{p}\log^2 p$.
\end{lem}\label{ZC theorem}
\begin{rem}
    If $\mathcal{C}$ is reducible over $\mathbb{F}_p$ as a union of a line and a conic, then take $\mathcal{C}$ in the statement of Lemma \ref{ZC theorem} to be the conic part. By the discussion in the proof of \mbox{Lemma \ref{elliptic curve satisfied}}, $\mathcal{C}$ can never be factored as a product of linear factors. 
\end{rem}
\begin{proof}
     The lemma is essentially an explicit version of \cite[Theorem 1]{cobeli}. The proof proceeds by Fourier expansion of the congruence-restricted interval indicators and separates the contribution of the zero frequency from the nonzero frequencies. The zero frequency gives the main term, while the nonzero frequencies are bounded using Bombieri-Weil estimate for the exponential sums along the curve, yielding the desired error term. One can see detailed proof in \cite[Lemma~9.1]{ours}.  
\end{proof}

We are now ready to prove Lemma \ref{sum of eight square pyramidal number}. The first step is to obtain an upper bound for a positive integer $m$ that can be written as a sum of eight square pyramidal numbers. 
\begin{lem}\label{main lemma}
    Let $m$ be a positive integer. Then, for any $m\geqslant \num{1.907e28}$, $m$ can be written as a sum of at most eight square pyramidal numbers. 
\end{lem}
\begin{proof} 
    We choose $t=3.17,\varepsilon=\delta=0.5$ and
\[\alpha = 0.5\in\bigg(\bigg(\frac{3t}{2}  - \frac{3\varepsilon}{2}-12\bigg)^{1/3}+\delta,\bigg(\frac{3t}{2}  + \frac{3\varepsilon}{2}-3\bigg)^{1/3}-\delta\bigg).\]
These parameters are carefully chosen to fulfill both Conditions 1 and 2. We first focus on Condition 2. \\\mbox{Let $t_1=t_2=\alpha$} and $g_1=g_2=\delta$.  Since  $(\alpha\pm\delta)\in [0,1]$, by Lemma \ref{Z-C theorem}, the number of Lehmer points $(\ell,r)$ that satisfying conditions laid out in \eqref{final conditions} is
\[F(p,\mathcal{C},\textbf{a},\textbf{b};\textbf{g}, \textbf{t})=\frac{\delta^2(p+1)}{4}+E^*,
\]
where $|E^*| <6.0009\sqrt{p}\log^2p$. Then, in order to have $F(p,\mathcal{C},\textbf{a},\textbf{b};\textbf{g}, \textbf{t})>0$, it suffices to have 
\begin{align}\label{Prime Requirement}
\delta^2(p+1)&>24.0036\sqrt{p}\log^2p.
\end{align}
When $m\geqslant \num{1.907e28}$, by \eqref{p cubed range}, we have 
\[
    \frac{m}{t-\varepsilon} \geqslant \frac{\num{1.907e28}}{2.67} \geqslant p^3,
\]
and thus, $p\geqslant \num{1.9257e9}$, which \mbox{satisfies \eqref{Prime Requirement}.} Therefore, when $m\geqslant \num{1.907e28}$, if we can find a prime $p$ such that 
\[\frac{1}{\sqrt[3]{t+\varepsilon}}\sqrt[3]{m}\leqslant p \leqslant \frac{1}{\sqrt[3]{t-\varepsilon}}\sqrt[3]{m},
\]
then the existence of $(\ell,r)$ satisfying \eqref{final conditions} is guaranteed by the above discussion, and Condition 2 will be met. In fact, Theorem 2 in Rosser and Schoenfeld's paper \cite{rosser1962} guarantees such a $p$. To fulfill Condition 1, it suffices to achieve the alternative form in \eqref{Alternative form of condition 1}. With the chosen $t,\varepsilon,\delta$ and $\alpha$, applying the lower and upper bounds \mbox{of $f(\ell)+f(r)$} in \eqref{f(l)+f(r) upper bound} and \eqref{flfr lower bound}, we see that when $p\geqslant \num{1.9257e9}$, the chosen parameters guarantee \mbox{us \eqref{Alternative form of condition 1}.} Thus, we conclude that any integer $m\geqslant \num{1.907e28}$ can be written as a sum of eight square pyramidal numbers.
\end{proof}

To this end, we can prove Lemma \ref{sum of eight square pyramidal number} by use of Lemma \ref{main lemma} and the following algorithm. 
\begin{proof}[Proof of Lemma \ref{sum of eight square pyramidal number}]
When $m\geqslant\num{1.907e28}$, by Lemma \ref{main lemma} the result is immediate, so assume that $m<\num{1.907e28}$. If $m$ is a square pyramidal number, then the theorem follows trivially. Otherwise, consider the following steps.\\   
\vspace{-0.75cm}
\subsection*{Step 1.} Let $c$ be a positive constant and $n$ be the integer such that $f(n)< m< f(n+1)$. If $m\in (f(n),f(n)+c)$, then let $m_1=m-f(n-1)$. Otherwise, let $m_1=m-f(n)$.
\subsection*{Step 2.} Take $m - m_1$ to be the first term in the sum. Reduce the problem to writing $m_1$ as a sum of one fewer square pyramidal numbers. 
\subsection*{Step 3.} Repeat Steps 1 and 2 with $m$ replaced by $m_1$ until the problem is reduced to writing the number as a sum of 4 square pyramidal numbers.\\
\indent We apply the above steps with $c=\num{1.3e7}$ to some positive integer $m < \num{1.907e28}$.\\  If $m\in [f(n)+c, f(n+1))$ for some $n\in \mathbb{N}$, then let $m_1=m-f(n)$. We deduce that 
\begin{align}
       & m_1 < f(n+1)-f(n)= n^2+2n+1.\label{m_1 bound}
\end{align}

Since $m<\num{1.907e28}$, 
\[
g(n) = 2n^3+3n^2+n<\num{1.907e28},
\]
which implies that $n<\num{3.854e9}.$ Using \eqref{m_1 bound}, we conclude that $m_1<\num{1.485e19}$. Alternatively, \mbox{when $m\in (f(n),f(n)+c)$,} let $m_{1} = m - f(n-1)$, and we obtain
\begin{align}
m_1 < f(n)+\num{1.3e7}-f(n-1) = n^2+\num{1.3e7}.\label{m_1 bound2}
\end{align}

Again, {since $m<\num{1.907e28}$}, we may conclude that $m_1<\num{1.4848e19}$. Therefore, in both cases, it suffices to write $\num{1.3e7}\leqslant m_1\leqslant \num{1.4848e19}$ as a sum of 7 square pyramidal numbers.

By iterating the above steps, we reduce the original problem to writing a positive integer $m_4$, such\\ \mbox{that $\num{1.3e7}\leqslant m_4\leqslant 15264785$,} as a sum of 4 square pyramidal numbers. This was confirmed by Algorithm I in \cite{github}. 

For positive integers less than $\num{1.3e7}$, we repeat step 1 and 2 three times with $c=10^5$ \mbox{and $m=\num{1.3e7}$.} With modified algorithm I from \cite{github}, we can again confirm that all integers in the \mbox{range $[10^5,\num{1.1244e9}]$} can be written as a sum of five square pyramidal numbers, and thus show that all integers in the \mbox{range $[10^5, \num{1.3e7}]$} can be written as a sum of at most eight square pyramidal numbers. Finally, for integers smaller \mbox{than $10^5$,} we check them using Algorithm II from \cite{github}.
\end{proof}
\section{Setup for Lemma \ref{Lemma: Representation}}\label{sec: initial setup}
For the remaining sections, we are devoted to proving Lemma \ref{Lemma: Representation}. In this section, we will employ the setup of the Hardy--Littlewood method to prove the general version of Lemma \ref{Lemma: Representation}, with $m$ equal to any arbitrarily large integer.
For $n\in\N$, recall in \eqref{defn: square pyramidal number} that the $n$-th square pyramidal number is defined by
\[
S_n = \frac{2n^3+3n^2+n}{6}.
\]

For $s \in \mathbb{N}$, let $\mathcal{C}_s(m)$ denote the number of ways of writing $m$ as the sum of $s$ square pyramidal numbers. \mbox{For $N \in \mathbb{N}$, $\alpha \in \mathbb{R}$,} consider the sum
\begin{align}\label{Vinogradov's Sum}
f_N(\alpha,S) = \sum_{n=1}^{N} e(\alpha S_n).
\end{align}

For ease of notation, we let $f(\alpha) = f_N(\alpha,S)$, with $N$ being clearly implied. We let
\begin{align}\label{Defining N}
N =  \left\lceil  (3m)^{\frac{1}{3}} \right\rceil +1, \quad P=N^\delta, 
\end{align}
where $\delta \in \mathbb{R}$ such that $N^{3\delta-3}<\frac{1}{2}$. For $1 \leqslant a \leqslant q \leqslant P$ and $(a, q)=1$, let
\[
\mathfrak{M}(q, a)=\left\{\alpha:|\alpha-a / q| \leqslant N^{\delta-3}\right\} .
\]
We call $\mathfrak{M}(q, a)$ the \textit{major arcs}, which are pairwise disjoint due to the restriction on $\delta$. \mbox{Let $\mathfrak{M}$} denote the union of the $\mathfrak{M}(q, a)$'s. Let $\mathfrak{U}=\left(N^{\delta-3}, 1+N^{\delta-3}\right]$ 
and define the \textit{minor arcs} \mbox{as $\mathfrak{m}=\mathfrak{U} \backslash \mathfrak{M}$.} Then Cauchy's integral formula yields
\begin{align}\label{Major+Minor Arcs}
     \mathcal{C}_s(m) = \int_{\mathfrak{M}} f(\alpha)^s e(-\alpha m) \, d \alpha+\int_{\mathfrak{m}} f(\alpha)^s e(-\alpha m) \, d \alpha.
\end{align}
 
 We will treat the minor arcs in Section \ref{Minor Arcs} and address the major arcs in Sections \ref{sec: Major Arcs}, \ref{sec: Major Arcs II} and \ref{sec: Major Arcs Sec III}.

\section{Minor Arcs}\label{Minor Arcs}
Our primary goal is to prove Lemma \ref{Minor Arcs : thm}, which obtains an explicit upper bound for the integral over the minor arcs $\mathfrak{m}$.
\begin{lem}\label{Minor Arcs : thm}
Let $f(\alpha)$ be as defined in \eqref{Vinogradov's Sum}. Then for $s \geqslant 9$ and $m \geqslant 10^{10}$, we have 
\[
\int_{\mathfrak{m}} f(\alpha)^s e(-m \alpha) \, d \alpha \leqslant 152 \cdot 13^{s-8} \cdot (\log m)^{\frac{s-8}{4}}m^{\frac{s}{3}-1-\frac{\delta(s-8)}{12}+\frac{0.53305(s-8)+6.3966}{1.2\log \log m}}.
\]
\end{lem}

 In this section, we prove some preliminary lemmas that are necessary in our treatment of the minor arcs. We also utilize results obtained by Basak and three of the authors \cite{ours}. For the sake of completeness, when appropriate, we will reproduce the proofs here and highlight the key differences in our arguments. 
\begin{lem}[{\cite[Lemma~3.6]{ours}}]\label{Weyl Inequality V2}
Let $\eta \geqslant 1$ be fixed and $X \geqslant e^3$. Let $\alpha_1, \alpha_2, \alpha_3 \in \mathbb{R}$ and suppose there \mbox{exists $a \in \mathbb{Z}$} and $q \in \mathbb{N}$ which satisfy $(a, q)=1$ and $\left|\alpha_3-a / q\right| \leqslant \eta q^{-2}$. Then
\[
\bigg \lvert \sum_{1 \leqslant x \leqslant X} e\left(\alpha_1 x+\ldots+\alpha_3 x^3\right) \bigg \rvert \leqslant 2X^{\frac{3}{4}}+8X^{1+\frac{0.53305}{\log (2\log X)}}\eta^{\frac{1}{4}}\left(q^{-1}+X^{-1}+q X^{-3}\right)^{\frac{1}{4}} (\log q)^{\frac{1}{4}}.
\] 
\end{lem}

\par

\begin{lem}\label{Inductive Lemma}
Let $f(\alpha)$ and $N$ be as defined in \eqref{Vinogradov's Sum} and \eqref{Defining N}, respectively. Suppose \mbox{ $1 \leqslant j \leqslant 3$.} Then, for $N \geqslant e^3$, one has
\[
\int_0^1|f(\alpha)|^{2^j} \, d \alpha \leqslant 152 N^{2^j-j+\frac{6.3966}{\log \log N}}.
\]
\end{lem}
\begin{proof} 
One can bound the moments of the exponential sums by repeated Weyl differencing and orthogonality. \mbox{For $j=1$,} orthogonality gives the diagonal count; for $j=2$ and $j=3$, differencing reduces the moment estimates to counting polynomial equations, where the zero-difference cases are bounded directly and the nonzero cases are controlled by divisor bounds. Detailed proof is analogous to \cite[Lemma~4.2]{ours}.
\end{proof}

We are now ready to prove the lemma about estimation on the minor arcs. 
\begin{proof}[Proof of Lemma \ref{Minor Arcs : thm}] 
We can write
\[
\begin{aligned}
\left|\int_{\mathfrak{m}} f(\alpha)^s e(-m \alpha) \, d \alpha\right| & \leqslant\left(\sup _{\alpha \in \mathfrak{m}}|f(\alpha)|\right)^{s-8} \int_0^1|f(\alpha)|^{8}  \, d  \alpha.
\end{aligned}
\]

Consider an arbitrary point $\alpha$ of $\mathfrak{m}$. By Dirichlet's Theorem (see \cite[Lemma 2.1]{vaughan1997hardy}), there exist $a, q$ \mbox{with $(a, q)=1$} and $q \leqslant N^{3-\delta}$, such that $|\alpha-a / q| \leqslant q^{-1} N^{\delta-3}$.

Since $\alpha \in \mathfrak{m} \subset\left(N^{\delta-3}, 1-N^{\delta-3}\right)$ it follows that $1 \leqslant a \leqslant q$. Therefore $q>N^\delta$, for otherwise $\alpha$ would be \mbox{in $\mathfrak{M}$.} Applying Lemma \ref{Weyl Inequality V2} with $\psi(n)=\alpha S_n$ and $\eta = 3$, we obtain
\begin{align}
\lvert f(\alpha) \rvert &\leqslant  2N^{\frac{3}{4}}+8 \cdot (3)^{\frac{1}{4}} N^{1+\frac{0.53305}{\log \log N}}\left(q^{-1}+N^{-1}+q N^{-3}\right)^{\frac{1}{4}} (\log N)^{\frac{1}{4}} \notag \\
&\leqslant  2N^{\frac{3}{4}}+8 \cdot (9)^{\frac{1}{4}} N^{1-\frac{\delta}{4}+\frac{0.53305}{\log \log N}} (\log N)^{\frac{1}{4}} \notag \\
& \leqslant 16  N^{1-\frac{\delta}{4}+\frac{0.53305}{\log \log N}} (\log N)^{\frac{1}{4}}. \label{f_alpha bound dode}
\end{align}

  By Lemma \ref{Inductive Lemma}, we have
\begin{align}
\int_0^1|f(\alpha)|^{8} \, d \alpha \leqslant 152 N^{5+\frac{6.3966}{\log \log N}}. \label{f_alpha Integral Bound dode}
\end{align}

Combining \eqref{f_alpha bound dode} and \eqref{f_alpha Integral Bound dode}, we find
\[
\begin{aligned}
\left|\int_{\mathfrak{m}} f(\alpha)^s e(-m \alpha) \, d \alpha\right| &  \leqslant 152 \left(16 N^{1-\frac{\delta}{4}+\frac{0.53305}{\log \log N}}\right)^{s-8} (\log N)^{\frac{s-8}{4}} N^{5+\frac{6.3966}{\log \log N}} \\
& \leqslant 152 \cdot 13^{s-8} \cdot (\log m)^{\frac{s-8}{4}}m^{\frac{s}{3}-1-\frac{\delta(s-8)}{12}+\frac{0.53305(s-8)+6.3966}{1.2\log \log m}},
\end{aligned}
\]
where the last line follows from transforming the estimate involving $N$ into $m$ using \eqref{Defining N}. For simplicity, we use the bound $(3m)^{1/3}<N<2m^{1/3}$ here.
This completes the proof of Lemma \ref{Minor Arcs : thm}.
\end{proof}
\section{Major Arcs: Preliminaries}\label{sec: Major Arcs}
Our primary goal in this section is to assert Lemma \ref{Approx 4} so that we effectively approximate $f(\alpha)$ \mbox{when $\alpha\in\mathfrak{M}$.} To do so, we first define
\begin{align}
A(t) &:=  \sum_{1\leqslant n\leqslant t} e\bigg (\frac{a}{q}S_n \bigg), \label{Definition 1} \\
\textrm{and} \quad V(q,a) &:= \sum_{1\leqslant n\leqslant 6q} e\bigg (\frac{a}{q}S_n \bigg) \label{Definition of V(q,a)}.
\end{align}

The main lemma of the section is as follows.
\begin{lem}\label{Approx 4}
For $N \geqslant 100$,
\[
\bigg \lvert \int_{\mathfrak{M}}f(\alpha)^s e(-\alpha m) \, d\alpha -\int_{\mathfrak{M}} \bigg (\frac{V(q,a)}{6q}\int_1^N e\bigg( \frac{t^3 \theta}{3} \bigg) \, dt \bigg)^s  e(-\alpha m) \, d\alpha \bigg \rvert \leqslant 50s N^{5 \delta+s-4}.
\]
\end{lem}

We achieve this bound by a series of successive lemmas. Let $\theta=\alpha-a/q$. Recalling from Section \ref{sec: initial setup} that since the $\mathfrak{M}(q, a)$ are pairwise disjoint,
\begin{align}\label{Major Arc Disjointness}
\int_{\mathfrak{M}} f(\alpha)^s e(-\alpha m)= \sum_{q\leqslant N^\delta}\sum\limits_{\substack{a=1 \\ (a,q)=1}}^q\int_{\mathfrak{M}(q,a)} f(\alpha)^se(-\alpha m) \, d\alpha.
\end{align} 

In what follows, by abuse of notation, we extend the definition of square pyramidal numbers to all real \mbox{numbers $t \in \R$} and write
\[
S_t = \frac{1}{3}t^3+\frac{1}{2}t^2+\frac{1}{6}t, \quad t \in \mathbb{R}.
\]

We obtain the following result by applying partial summation. 
\begin{lem}\label{Partial Summation}
Let $\alpha \in \mathfrak{M}(q,a)$ and $\theta = \alpha - a/q$. Then for $f(\alpha)$ defined in \eqref{Vinogradov's Sum},
\begin{align}
 f(\alpha) &=A(N)e(\theta S_N)-2\pi i\theta\int_{1}^NA(t)\left (t^2+t+\frac{1}{6}\right )e(\theta S_t) \, dt.
\end{align}
\end{lem}

\begin{lem}\label{Congruence}
If $(6,q)=1$, then  $S_n\equiv S_{n+q}$ mod $q$. Otherwise, if $(2,q)=1$, then  $S_n\equiv S_{n+3q}$ mod $q$; $(3,q)=1$, then  $S_n\equiv S_{n+2q}$ mod $q$. Finally, if $6\mid q$, then $S_n\equiv S_{n+6q}$ mod $q$.
\end{lem}
\begin{proof}
The congruence relation $S_n\equiv S_{n+6q} \bmod q$ always holds true. Note that 
\begin{align}
    S_{n+q}=\frac{1}{3}(n+q)^3+\frac{1}{2}(n+q)^2+\frac{1}{6}(n+q)&\equiv \frac{1}{3}n^3+\frac{1}{2}n^2+\frac{1}{6}n+\frac{1}{3}q^3+\frac{1}{2}q^2+\frac{1}{6}q\notag\\
    &\equiv S_n+\frac{1}{3}q^3+\frac{1}{2}q^2+\frac{1}{6}q\hspace{0.1cm}\bmod q.\label{modulo q}
\end{align}

If $(6,q)=1$, then the coefficients of the last three terms in \eqref{modulo q} are all invertible modulo $q$, and therefore, we see that $S_{n+q}\equiv S_n\bmod q$.
Similarly, if $(2,q)=1$, then $S_{n+3q}\equiv S_n\bmod q$; and if $(3,q)=1$, \mbox{then $S_{n+2q}\equiv S_n\bmod q$.} The result then follows.
\end{proof}
\begin{lem}\label{Approx 1}
For all $1 \leqslant t \leqslant N$, with $A(t)$ defined in \eqref{Definition 1}, 
\begin{align*}
   \bigg \lvert A(t) - \frac{V(q,a)}{6q} t \bigg \rvert \leqslant 6q.
\end{align*}
\end{lem}
\begin{proof}
Writing
\[
A(t) = \sum_{n=1}^{6q} e\bigg (\frac{a}{q}S_n \bigg)\floor*{\frac{t}{6q}} +\sum_{n = \floor*{\frac{t}{6q}}6q+1}^t e \bigg (\frac{a}{q}S_n \bigg)
\]
and using Lemma \ref{Congruence}, we obtain the desired result. 
\end{proof} 
\begin{lem}\label{Approx 2}
Let $\alpha \in \mathfrak{M}(q,a)$ and $\theta = \alpha - a/q$.
Then
\[
\bigg \lvert f(\alpha)-\frac{V(q,a)}{6q}\int_1^N e(\theta S_t)\, dt \bigg \rvert \leqslant 6q+1+5q\pi\theta N^3.
\]
\end{lem}
\begin{proof} 
 By use of Lemma \ref{Approx 1}, we obtain
\begin{align*}
    \bigg \lvert A(N)e(\theta S_N)-\frac{V(q,a)}{6q} N e(\theta S_N)\bigg \rvert 
   &\leqslant 6q,\\
\textrm{and} \quad \bigg \lvert 2\pi i\theta \int_{1}^NA(t)\bigg(t^2+t+\frac{1}{6} \bigg)e(\theta S_t)\, dt&-2\pi i\theta\int_{1}^N\frac{V(q,a)}{6q}t\bigg(t^2+t+\frac{1}{6} \bigg )e(\theta S_t)\, dt \bigg \rvert \notag \\
    &\leqslant 2\pi\theta\int_{1}^N6q \bigg(t^2+t+\frac{1}{6}  \bigg )\, dt \leqslant 5q\pi\theta N^3.
\end{align*}

Therefore, employing the triangle inequality and Lemma \ref{Partial Summation},
\begin{align} \label{Approx 2 Step 1}
\bigg \lvert f(\alpha)-\bigg(\frac{V(q,a)}{6q}Ne(\theta S_N)-2\pi i\theta\int_{1}^N&\frac{V(q,a)}{6q}t\bigg (t^2+t+\frac{1}{6} \bigg )e(\theta S_t)\, dt \bigg)\bigg \rvert \notag\\
&\leqslant 6q+5q\pi\theta N^3.
\end{align}
Applying integration by parts, we have
\begin{align}
\frac{V(q,a)}{6q} N e(\theta S_n)&-2\pi i\theta\int_{1}^N\frac{V(q,a)}{6q}t\bigg (t^2+t+\frac{1}{6} \bigg )e(\theta S_t) \, dt \notag \\
&=\frac{V(q,a)}{6q}\bigg (e(\theta)+\int_1^N e(\theta S_t) \, dt\bigg ).\label{Approx 2 Step 2}
\end{align}

Putting together \eqref{Approx 2 Step 1} and \eqref{Approx 2 Step 2} and  trivially bounding $V(q,a)$, we achieve the final bound
\[
\bigg \lvert f(\alpha)-\frac{V(q,a)}{6q}\int_1^N e(\theta S_t)\, dt \bigg \rvert \leqslant 6q+1+5q\pi\theta N^3.
\]
\end{proof}
\begin{rem}
    Observe that in Lemma \ref{Approx 2}, we use the trivial bound $|V(q,a)|\leq 6q$ instead of the sharp bound from Lemma \ref{Weyl Inequality V2}. This trivial bound suffices here because $q$ is later summed only up to $N^\delta$, so the resulting contribution is absorbed when bounding other larger error terms, and therefore can be estimated more crudely.
\end{rem}
\begin{lem}\label{Approx 3}
Let $N \geqslant 100, \alpha \in \mathfrak{M}(q,a)$ and $\theta = \alpha - a/q$.
Then
\[
\bigg \lvert f(\alpha)-\frac{V(q,a)}{6q}\int_1^N e\bigg( \frac{t^3 \theta}{3} \bigg) \, dt \bigg \rvert \leqslant 6q+1+5q\pi\theta N^3+3\pi N^\delta.
\]    
\end{lem}
\begin{proof} 

By Lemma \ref{Approx 2}, it is sufficient to show
\[
\bigg \lvert \int_1^Ne(\theta S_t) \, dt-\int_1^N e\bigg( \frac{t^3 \theta}{3} \bigg) \, dt \bigg \rvert \leqslant 3\pi N^\delta.
\]

To see this, we write
\begin{align}\label{Approx 3 Step 1}
\bigg \lvert \int_1^Ne(\theta S_t) \, dt-\int_1^N e\bigg( \frac{t^3 \theta}{3} \bigg) \, dt \bigg \rvert \leqslant \int_1^N \bigg \lvert e\bigg(\frac{t^2 \theta}{2}+\frac{t}{6}\theta\bigg)-1\bigg \rvert \, dt.
\end{align}

Since $|\theta|\leqslant N^{\delta-3}$ and $1 \leqslant t \leqslant N,$ we have for any $N \geqslant 100$,
\[
\bigg \lvert 2\pi\theta\bigg(\frac{t^2}{2}+\frac{t}{6}\bigg) \bigg \rvert \leqslant 2\pi N^{\delta-1} \leqslant 1,
\]
and therefore, by Taylor expansion,
\begin{align}
\bigg \lvert e\bigg(\frac{t^2 \theta}{2}+\frac{t}{6}\theta\bigg)-1\bigg \rvert 
\leqslant 3\pi N^{\delta-1}. \label{Approx 3 Step 2}
\end{align}

Substituting \eqref{Approx 3 Step 2} into \eqref{Approx 3 Step 1} and integrating, we acquire the desired result.
\end{proof}
\begin{proof}[Proof of Lemma \ref{Approx 4}] 
An application of the polynomial factorization combining with Lemma \ref{Approx 3} shows that \\\mbox{if $\alpha \in \mathfrak{M}(q,a)$} and $\theta = \alpha-a/q$, then
\begin{align*}
\bigg \lvert f(\alpha)^s-\bigg (\frac{V(q,a)}{6q}\int_1^N e\bigg( \frac{t^3 \theta}{3} \bigg) \, dt \bigg)^s \bigg \rvert \leqslant s N^{s-1}(6q+1+5q\pi\theta N^3+3\pi N^\delta).
\end{align*}
\begin{align}
\bigg \lvert \int_{\mathfrak{M}}f(\alpha)^s &e(-\alpha m) \, d\alpha -\int_{\mathfrak{M}} \bigg (\frac{V(q,a)}{6q}\int_1^N e\bigg( \frac{t^3 \theta}{3} \bigg) \, dt \bigg)^s  e(-\alpha m) \, d\alpha \bigg \rvert \notag \\
& \leqslant \sum_{q\leqslant N^\delta}\sum\limits_{\substack{a=1 \\ (a,q)=1}}^q\int_{-N^{\delta-3}}^{N^{\delta-3}}s((6q+1)N^{s-1}+5q\pi\theta N^{s+2}+3\pi N^{\delta+s-1}) \, d\theta \notag\\
&\leqslant 2s\sum_{q\leqslant N^\delta}\sum\limits_{\substack{a=1 \\ (a,q)=1}}^q(6N^{s+2\delta-4}+3\pi N^{s+3\delta-4}+3\pi N^{2\delta+s-4}) \leqslant 50sN^{5\delta+s-4},\label{4.10 approximation}
\end{align}
which completes the proof.
\end{proof}

\section{Major Arcs : Part I}\label{sec: Major Arcs II} 
Suppose $s \geqslant 9$. Let $\alpha \in \mathfrak{M}(q,a), \theta = \alpha - a / q$ and define
\begin{align}
\mathfrak{S}(m, Q) &: =\sum_{q\leqslant Q}\sum\limits_{\substack{a=1 \\ (a,q)=1}}^q\bigg(\frac{V(q,a)}{6q}\bigg)^s e\bigg(-\frac{am}{q}\bigg), \label{S (m,Q) definition} \\
v(\theta) &: = \int_1^N e\bigg (\frac{t^3 \theta}{3} \bigg) \dd t, \label{v theta definition} \\
\textrm{and} \quad J^*(m) &:= \int_{-N^{\delta-3}}^{N^{\delta-3}} \left(\left(\frac{1}{3}\right)^{1/3}v(\theta)\right)^se(-\theta m) \dd\theta. \label{J* definition}
\end{align}

Furthermore, we let 
\begin{align}
\mathcal{I}_{s}^*(m)&:=\int_{\mathfrak{M}} \bigg (\frac{V(q,a)}{6q}\int_1^N e\bigg( \frac{t^3 \theta}{3} \bigg) \dd t \bigg)^s  e(-\alpha m) \dd\alpha \notag\\
&=\sum_{q\leqslant N^\delta}\sum\limits_{\substack{a=1 \\ (a,q)=1}}^q\bigg(\frac{V(q,a)}{6q}\bigg)^s e\bigg(-\frac{am}{q}\bigg)\int_{-N^{\delta-3}}^{N^{\delta-3}} v(\theta)^se(-\theta m) \dd\theta \notag\\
&=3^{s/3}\mathfrak{S}(m,N^\delta)J^*(m) \label{Approximating Major Arc Integral}.
\end{align}

We will approximate the major-arc integral by $\mathcal{I}_{s}^*(m)$. To do so, we will estimate $\mathfrak{S}(m,N^\delta)$ and $J^*(m)$ separately. We will attain this in this section, Section \ref{sec: Major Arcs II}, and the next section, \mbox{Section \ref{sec: Major Arcs Sec III}}, respectively.

\subsection{Completing the Singular Series} We first complete the series $\mathfrak{S}(m,N^\delta)$. Define
\begin{align}
    \mathfrak{S}(m) &:= \sum_{q=1}^\infty V(q)\label{definition of S(m)}, \quad \textrm{where} \\
    \label{V Definition}
    V(q) &:= \sum\limits_{\substack{a=1 \\ (a,q)=1}}^q\bigg(\frac{V(q,a)}{6q}\bigg)^se\bigg(-\frac{am}{q}\bigg),
\end{align}
and $V(q,a)$ is given by \eqref{Definition of V(q,a)}. We first show that $V(q)$ is multiplicative.

\begin{lem}\label{V Multiplicativity 2}
The function $V(q)$, as defined in \eqref{V Definition} is multiplicative.
\end{lem}
\begin{proof}

 The complete exponential sums factor under coprime moduli by the Chinese remainder theorem and the periodity of $S_n$, giving $V(qr, ar+bq)=\frac{1}{6}V(q,a)V(r,b)$; Then summing the factorization over reduced residue classes, the normalized exponential factor splits into the product of the $q$-part and the $r$-part, and $V(qr)$ factors as \mbox{$V(q)V(r)$, }concluding the proof. One can see detailed analogous proof in \cite[Lemma~6.2]{ours}.
\end{proof}
\begin{lem}\label{Singular Series Extension}
Let $s \geqslant 9$ and $N^{\delta} \geqslant e^{e^{45}}$. Then $|\mathfrak{S}(m)|\leqslant e^{e^{46}}$. Furthermore, 
\begin{align*}
|\mathfrak{S}(m)-\mathfrak{S}(m, N^\delta)| \leqslant \frac{ 6^s}{\left(\frac{5s}{21}-2\right) N^{\left(\frac{5s}{21}-2\right)\delta}}.
\end{align*}
\end{lem}
\begin{proof}
We begin by applying Lemma \ref{Weyl Inequality V2} to evaluate $V(q,a)$. We cover all values of $a$ and $q$ by separating into cases $(a,3q)=1$ and $(a,3q)\neq 1$.  In Lemma \ref{Weyl Inequality V2}, it suffices to pick $\eta = 1$. We then have
\begin{align*}
|V(q,a)|
&\leqslant 8 q^{\frac{3}{4}}+26  q^{\frac{3}{4}+\frac{0.53305}{\log (2\log 6q)}+ \frac{\log \log q}{4\log q}} .
\end{align*}

When $q \geqslant e^{e^{45}}$,
\[
\frac{0.53305}{\log (2\log 6q)}+ \frac{\log \log q}{4\log q} \leqslant \frac{1}{84},
\]
implying that
\[
|V(q,a)| \leqslant 8q^{\frac{3}{4}}+26 q^{\frac{16}{21}} \leqslant 34 q^{\frac{16}{21}}.
\]

Thus, $ |V(q)| \leqslant 6^s q^{1-\frac{5s}{21}}$ for $q \geqslant e^{e^{45}}$. Since $s \geqslant 9$, $\mathfrak{S}(m)$ must converge absolutely and uniformly with respect to $m$. Therefore,
\begin{align*}
|\mathfrak{S}(m)| 
&\leqslant\sum_{q=1}^{\lfloor e^{e^{45}} \rfloor}|V(q)|+ 6^s\int_{e^{e^{45}}}^\infty x^{1-\frac{5s}{21}}\dd x
    \leqslant e^{e^{46}}.
\end{align*}

Moreover, when $N^{\delta} \geqslant e^{e^{45}}$,
we deduce that
\begin{align*}
    |\mathfrak{S}(m)-\mathfrak{S}(m, N^\delta)| \leqslant 6^s\int_{N^{\delta}}^\infty x^{1-\frac{5s}{21}} \dd x \leqslant  \frac{ 6^s}{\left(\frac{5s}{21}-2\right) N^{\left(\frac{5s}{21}-2\right)\delta}}.
\end{align*}
\end{proof}

\subsection{Counting Solutions to Arithmetic Congruences}\label{subsec: Arithmetic Congruences} 
The Euler product
\begin{align*}
    \mathfrak{S}(m) =\prod_{p\text{ prime}}\sum_{k=0}^\infty V(p^k)=\prod_{p\text{ prime}} (1+V(p)+V(p^2)+\cdots)
\end{align*}
emerges from the fact that $\mathfrak{S}(m)$ converges absolutely by Lemma \ref{Singular Series Extension} and the multiplicativity of $V(q)$. Later in this section, we will utilize this form of $\mathfrak{S}(m)$ after further evaluating sums involving $V(q)$.

Suppose $1\leqslant n_i\leqslant t$. Let $\mathcal{M}_m(t,q)$ be the number of solutions of the congruence equation
\begin{align}\label{f(x) formulas}
g(n_1)+g(n_2)+\cdots+g(n_s)\equiv m \bmod q,
\end{align}
where $g(x)=\frac{1}{3}x^3+\frac{1}{2}x^2+\frac{1}{6}x$. By abuse of notation, we write $\mathcal{M}_m(q,q) = \mathcal{M}_m(q)$.

One way to estimate the number of solutions is to apply the Lang--Weil theorem \cite{langweil} on counting points on varieties over finite fields. However, we will proceed differently to obtain a power saving larger than a full factor of $p$.


\begin{lem}\label{V Multiplicativity 3}
For $q \in \mathbb{N}$, 
\[
\sum_{d\mid q}V(d)=q^{1-s}6^{-s}\mathcal{M}_m(6q,q). 
\]
\end{lem}
\begin{proof}
Using the same method as in \cite[Lemma~6.4]{ours} with writing
\begin{align*}
\mathcal{M}_m(6q,q)=\frac{1}{q}\sum_{r=1}^q\sum_{n_1=1}^{6q}\sum_{n_2=1}^{6q}\cdots\sum_{n_s=1}^{6q}e(r(g(n_1)+g(n_2)+\cdots+g(n_s)-m)/q),
\end{align*}
we obtain that
\begin{align*}
    \mathcal{M}_m(6q,q) = q^{s-1}6^s\sum_{d\mid q}V(d).
\end{align*}

Thus, we have $\sum_{d\mid q}V(d)=q^{1-s}6^{-s}\mathcal{M}_m(6q,q)$.
\end{proof}
By choosing $q=p^k$ and applying Lemma \ref{V Multiplicativity 3}, we obtain
\begin{align*}
    \mathfrak{S}(m) &=\prod_{p\textnormal{ prime}}\sum_{k=0}^\infty V(p^k)=\prod_{p\textnormal{ prime}} \lim_{k\rightarrow\infty} 6^{-s}p^{k(1-s)}\mathcal{M}_m(6p^k,p^k).
\end{align*}

Let 
\begin{align}\label{T Definition}
T_m(p):= \lim_{k\rightarrow\infty} p^{k(1-s)}\mathcal{M}_m(p^k).
\end{align}

Since $\mathcal{M}_m(6q,q) = 6^s \mathcal{M}_m(q)$, 
\begin{align}\label{Crucial Limit Result}
\mathfrak{S}(m) = \prod_{p\text{ prime}} \lim_{k\rightarrow\infty} p^{k(1-s)}\mathcal{M}_m(p^k)= \prod_{p\text{ prime}} T_m(p).
\end{align}

Thus, the problem of estimating $\mathfrak{S}(m)$, is reduced to evaluating $T_m(p)$.
\begin{lem}\label{T_m(p) complete bound}
For any $s \geqslant 9$ and any prime $p$, we have
\begin{align*}
\left|T_m(p)-1\right| \leqslant e^{se^{89}}\left(\frac{p - 1}{p}\right)\cdot \left( \frac{p^{1 - \frac{5s}{21}}}{1 - p^{1 - \frac{5s}{21}}}  \right).
\end{align*}
\end{lem}
\begin{proof} 
We consider cases where $p = 2$, $p =3$, and $p \geqslant 5$.

\noindent\textbf{Case 1 : $p \geqslant 5$.}  Let $\tilde{g}(x) = 2x^3+3x^2+x$. Since $(6, p)=1$, any solution $(n_1,n_2,\dots,n_s)$ \mbox{of \eqref{f(x) formulas}} is also a solution to
\begin{align}
\tilde{g}(n_1)+\tilde{g}(n_2)+\cdots+\tilde{g}(n_s)\equiv 6m \bmod p^k.\label{g(x) formula}
\end{align}

Thus, we may write
\begin{align}
\mathcal{M}_m(p^k)&=\frac{1}{p^k}\sum_{t=1}^{p^k}\sum_{n_1=1}^{p^k}\sum_{n_2=1}^{p^k}\cdots\sum_{n_s=1}^{p^k}e(t(\tilde{g}(n_1)+\tilde{g}(n_2)+\cdots+\tilde{g}(n_s)-6m)/p^k)\notag \\
    &=p^{(s-1)k}+\frac{1}{p^k}\sum_{t=1}^{p^k-1}e\bigg(-\frac{6mt}{p^k}\bigg)\bigg(\sum_{x=1}^{p^k}e\bigg(\frac{t\tilde{g}(x)}{p^k}\bigg)\bigg)^s, \label{Primes > 7 Step 1}
\end{align}
where the second term in the right-hand side can be rewritten as 
\begin{align}\label{Primes > 7 Step 2}
    p^{(s-1)k}\sum_{r=1}^k p^{-rs}\sum\limits_{\substack{b=1 \\ (b,p)=1}}^{p^r}e\bigg(\frac{-6mb}{p^r}\bigg) \bigg(\sum_{x=1}^{p^r}e\bigg(\frac{b\tilde{g}(x)}{p^r}\bigg)\bigg)^s.
\end{align}
Applying Lemma \ref{Weyl Inequality V2} with $X = q = p^{r}$, we have
\begin{align} \left|\sum_{x=1}^{p^r}e\bigg(\frac{b\tilde{g}(x)}{p^r}\bigg)\right|& \leqslant 2p^{\frac{3r}{4}}+12p^{\frac{3r}{4}+\frac{0.53305r}{\log (2r \log p)}+\frac{\log (r \log p)}{4\log p}}. 
\end{align}
Observe that for any $r\geq 1$ and $p\geq e^{e^{45}}$, \begin{align}\label{finding p values}
\frac{0.53305r}{\log (2r \log p)}+\frac{\log (r \log p)}{4\log p} \leqslant \frac{r}{84}.
\end{align}
One may also check that \eqref{finding p values} holds for all $p \geqslant 5$ if $r\geqslant e^{45}$. Furthermore, for $5 \leqslant p < e^{e^{45}}$ and $1 \leqslant r \leqslant e^{45}$, trivially we have
\begin{align*}
    \frac{0.53305r}{\log (2r \log p)}+\frac{\log (r \log p)}{4\log p} < e^{45}.
\end{align*}
Therefore, for any prime $p \geqslant5$ and all $r \in N$, we conclude that
\begin{align} \left|\sum_{x=1}^{p^r}e\bigg(\frac{b\tilde{g}(x)}{p^r}\bigg)\right|& \leqslant e^{e^{89}} p^{\frac{16r}{21}}.\label{p>=7 Weyl bound}
\end{align}

\noindent\textbf{Case 2 : $p=2$.} Following similar reasoning to Case 1, 
\begin{align*}
\mathcal{M}_m(2^k)
    &=2^{(s-1)k}+\frac{1}{2^k}\sum_{t=1}^{2^k-1}e\bigg(-\frac{3mt}{2^k}\bigg)\bigg(\sum_{x=1}^{2^k}e\bigg(\frac{t\tilde{g}(x)}{2^{k+1}}\bigg)\bigg)^s.
\end{align*}

Therefore,
\begin{align*}
\left|\mathcal{M}_m(2^k)-2^{(s-1)k}\right|& 
\leqslant 2^{(s-1)k}\sum_{r=1}^k 2^{-rs}\sum\limits_{\substack{b=1 \\ (b,2)=1}}^{2^r} \left|\sum_{x=1}^{2^r}e\bigg(\frac{b\tilde{g}(x)}{2^{r+1}}\bigg)\right|^s.
\end{align*}

Applying Lemma \ref{Weyl Inequality V2} with $X=2^r$ and $q=2^{r+1}$, we have for any $r \in \mathbb{N}$,
\begin{align}
\left|\sum_{x=1}^{2^r}e\bigg(\frac{b\tilde{g}(x)}{2^{r+1}}\bigg)\right|
\leqslant 2^{e^{44}}\cdot 2^{\frac{16r}{21}}.\label{p=2 Weyl bound}
\end{align}
\noindent\textbf{Case 3 : $p=3$.} This case will follow exactly like Case 2 when substituting $p=2$ by $p=3$ throughout the proof.

Combining \eqref{T Definition}, \eqref{Primes > 7 Step 1} and \eqref{Primes > 7 Step 2}, for $p\geqslant 5$, we write
\begin{align}    T_m(p)=1+\sum_{r=1}^\infty p^{-rs}\sum\limits_{\substack{b=1 \\ (b,p)=1}}^{p^r}e\bigg(\frac{-6mb}{p^r}\bigg)\bigg(\sum_{x=1}^{p^r}e\bigg(\frac{b\tilde{g}(x)}{p^r}\bigg)\bigg)^s. \notag
\end{align}

Furthermore,
\begin{align*}    
T_m(2)&=1+\sum_{r=1}^\infty 2^{-rs}\sum\limits_{\substack{b=1 \\ (b,2)=1}}^{2^r}e\bigg(\frac{-3mb}{2^r}\bigg)\bigg(\sum_{x=1}^{2^r}e\bigg(\frac{b\tilde{g}(x)}{2^{r+1}}\bigg)\bigg)^s, \\
T_m(3)&=1+\sum_{r=1}^\infty 3^{-rs}\sum\limits_{\substack{b=1 \\ (b,3)=1}}^{3^r}e\bigg(\frac{-2mb}{3^r}\bigg)\bigg(\sum_{x=1}^{3^r}e\bigg(\frac{b\tilde{g}(x)}{3^{r+1}}\bigg)\bigg)^s .
\end{align*}

Putting together \eqref{p>=7 Weyl bound} and \eqref{p=2 Weyl bound}, along with the above two series representations for $T_m(p)$, we deduce that for $s \geqslant 9$ and any prime $p$, we conclude
\begin{align}
\left|T_m(p)-1\right|&\leqslant\sum_{r=1}^\infty p^{-rs}\sum\limits_{\substack{b=1 \\ (b,p)=1}}^{p^r} e^{se^{89}}p^{\frac{16rs}{21}}\notag 
= e^{se^{89}}\bigg(\frac{p - 1}{p}\bigg)\cdot \bigg( \frac{p^{1 - \frac{5s}{21}}}{1 - p^{1 - \frac{5s}{21}}} \bigg) \notag.
\end{align}
\end{proof}

Consider the following congruence equation where $p \geqslant 5$,
\begin{align}\label{Derivative Equation}
x^2+x+\frac{1}{6}\equiv 0 \bmod p.
\end{align}

This equation has at most two solutions; call them $\alpha$ and $\beta$. We call a solution to \eqref{f(x) formulas mod p}, $(\delta_1,\delta_2, \dots, \delta_s)$, ``good" if there exists at least some $j \in \{1,2,\dots,s\}$ such that $\delta_j\not\in\{\alpha,\beta\}$.
\begin{lem}\label{hensel's lemma lifting}
Let $p \geqslant 5$ be a prime, $k \geqslant 2$ and consider the congruence equation
\begin{align}\label{f(x) formulas Hensel Lifting}
g(n_1)+g(n_2)+\cdots+g(n_s)\equiv m \bmod p^k,
\end{align}
where $g(x)=\frac{1}{3}x^3+\frac{1}{2}x^2+\frac{1}{6}x$ and $1\leqslant n_i\leqslant p^k$ for all $i \in \{1,2,\dots,s\}$. Then every good solution of \eqref{f(x) formulas mod p} can be lifted uniquely to $p^{(k-1)(s-1)}$ solutions of \eqref{f(x) formulas Hensel Lifting}. 
\end{lem}
\begin{proof} 
Observe that equation \eqref{Derivative Equation} corresponds to the derivative of $g(x)$. Therefore, the lemma follows directly from Hensel's lemma.  
\end{proof}
\begin{rem}\label{Hensel Remark}
Suppose that in \eqref{f(x) formulas mod p} all but $r$ of the $n_i$'s are fixed. Then the number of solutions to \eqref{f(x) formulas mod p} which are not good is at most $2^r$.
\end{rem}
\begin{lem}\label{T_m(p) lower bound}
For $s\geqslant 9$ and a prime $p$,
\[T_m(p)> \max\bigg\{p^{1-s},1-\frac{2}{\sqrt{p}}-\frac{3}{p}\bigg\}.\]
\end{lem}
\begin{proof}
 We consider the following cases. \\
\noindent \textbf{Case 1 : $p\geqslant 11$.} 
For $p \geqslant 11$,  By Lemma \ref{elliptic curve satisfied}, one can obtain an effective lower bound for $\mathcal{M}_m(p)$. From Remark \ref{Hensel Remark}, there are at most $4p^{s-2}$ bad solutions where Hensel's lifting may fail. After removing these bad solutions, one can lift the remaining good solutions to solutions modulo $p^k$ using Lemma \ref{hensel's lemma lifting}. This yields the bound
\begin{align*}
T_m(p)\geqslant 1-\frac{2}{\sqrt{p}}-\frac{3}{p}.
\end{align*}

\noindent \textbf{Case 2 : $p=2$.} Let $B = g(n_{2}) + \cdots + g(n_{s}) - m$. 
 Consider the congruence equation
\begin{align}\label{1 Variable}
g_0(x):=\frac{x^3}{3} + \frac{x^2}{2} + \frac{x}{6} + B \equiv 0 \bmod 2.
\end{align}

Taking $x=6n$ for some integer $n$, $g_0(6n) = 72n^3+18n^2+n+B$. Note that $n\equiv -B\bmod 2$, and \mbox{therefore, $g_0(6n)\equiv 0 \bmod 2$.} Moreover, $g_0'(6n)\equiv 1\bmod 2$. Thus, using Hensel's lemma, any solution of the form $x=6n$ such that $n\equiv -B\bmod 2$ is always a ``good" solution.

Applying \mbox{Lemma \ref{hensel's lemma lifting}}, we have $\mathcal{M}_m(2^k)\geqslant 2^{(k-1)(s-1)},$
implying,
\begin{align*}
T_m(2) \geqslant \lim_{k \rightarrow \infty}2^{k(1-s)}2^{(k-1)(s-1)}=2^{1-s}>0.
\end{align*}
\textbf{Case 3 : $p=3,5,7$.} The arguments for $p=3,5,7$ follow analogous reasoning to Case 2.
\end{proof}
\subsection{A lower bound for $\mathfrak{S}(m)$}
We now establish a lower bound for $\mathfrak{S}(m)$.
\begin{lem}\label{lower bound for S(m)}
We have $\mathfrak{S}(m)>0$. More precisely,
\begin{align}\label{Sigma Bound}
 \mathfrak{S}(m) \geqslant \num{1.56e-3}\cdot210^{1-s}\cdot 2^{\frac{42z}{42-5s}} \exp\bigg(-\frac{z \log 2}{\log z}\bigg(1+\frac{1.2762}{\log z}\bigg)\bigg),
\end{align}
where $z=(2 e^{se^{89}}+1)^{\frac{21}{5s - 21}}$.
\end{lem}
\begin{proof} 
Recall the formula of $\mathfrak{S}(m)$ in \eqref{Crucial Limit Result}. We split the estimation into small-prime and large-prime parts. For large prime, Lemma \ref{T_m(p) complete bound} shows that $T_m(p)$ is close to $1$, so a Taylor expansion gives a lower bound \mbox{for $\prod_{p>z}T_m(p)$.} For small primes, Lemma \ref{T_m(p) lower bound} gives explicit lower bound for each prime $p$, sogether with Dusart's upper bound for $\pi(x)$ (see \cite[Corollary~5.2]{dusart2010estimates}) gives the lower bound. Detailed analgous proof can be found in \cite[Lemma~6.10]{ours}.
\end{proof}
\section{Major Arcs : Part II}\label{sec: Major Arcs Sec III}  Here we estimate the singular integral $J^{*}(m)$ given by \eqref{J* definition}. We assume $N^{\delta} \geqslant e^{e^{45}}$.
Using $N$ from \eqref{Defining N}, let $N_0 =\frac{1}{3}N^3$. Define
\begin{align}
    v_1(\theta):= \frac{1}{3}\sum_{1 \leqslant n \leqslant N_0} n^{-2/3}e(\theta n) \quad \textrm{and} \quad v_2(\theta) := \int_0^{N_0^{1/3}} e(\theta t^3) \dd t.\label{v_1 definition}
\end{align}
\begin{lem}\label{v_3(theta)-v_1(theta) bound lemma}
Let $v(\theta)$ and $v_1(\theta)$ be as defined in \eqref{v theta definition} and \eqref{v_1 definition}, respectively. Then
\[\bigg \lvert v_1(\theta)-\left(\frac{1}{3}\right)^{1/3}v(\theta)\bigg \rvert \leqslant 7N^\delta.
\]
\end{lem}
\begin{proof}
We have
\begin{align}
   \bigg \lvert v_2(\theta) -\left(\frac{1}{3}\right)^{1/3}v(\theta) \bigg \rvert &\leqslant \int_0^{\sqrt[3]{\frac{1}{3}}} \left|e(\theta t^3)\right| \dd t=\left ( \frac{1}{3}\right)^{1/3}\label{|v_2(theta)-v_1(theta)|}.
\end{align}

By partial summation, we find that 
\begin{align}\label{Partial Summation v_1}
    v_1(\theta)=e(N_0\theta)\bigg(\frac{1}{3}\sum_{1 \leqslant n \leqslant N_0} n^{-2/3}\bigg)-2\pi i\theta\int_1^{N_0} \bigg(\frac{1}{3}\sum_{1 \leqslant n \leqslant t} n^{-2/3}\bigg)e(\theta t) \dd t.
\end{align}
and
\begin{align}\label{By Parts v_2}
v_2(\theta) = e(N_0\theta)N_0^{1/3}-2\pi i\theta\int_0^{N_0} t^{1/3} e(\theta t) \dd t.
\end{align} 

Thus, we deduce that
\begin{align}
\left|v_1(\theta)-v_2(\theta) \right| &\leqslant \bigg \lvert \frac{1}{3}\sum_{1 \leqslant n \leqslant N_0} n^{-2/3}-N_0^{1/3}\bigg \rvert +2\pi|\theta|\bigg (1+ \int_1^{N_0} \bigg \lvert t^{1/3} - \bigg(\frac{1}{3}\sum_{1 \leqslant n \leqslant t} n^{-2/3}\bigg)\bigg \rvert \dd t  \bigg) \notag\\
&\leqslant3+2\pi|\theta|(3N_0-2) \label{v_1-v_2}.
\end{align}

Recall that $|\theta|<N^{\delta-3}$. Putting together \eqref{|v_2(theta)-v_1(theta)|} and \eqref{v_1-v_2}, we have
\begin{align*}
\bigg \lvert v_1(\theta)-\left(\frac{1}{3}\right)^{1/3}v(\theta)\bigg \rvert \leqslant  3+\pi(6N_0-4)N^{\delta-3}+\sqrt[3]{1/3} \leqslant 7N^\delta.
\end{align*}
\end{proof}
\subsection{Completing the Singular Integral}  We begin by considering the integral
\begin{align}
J_1^*(m) &:= \int_{-N^{\delta-3}}^{N^{\delta-3}} v_1(\theta)^se(-\theta m) \dd \theta. \label{J_1* definition}
\end{align}

By Lemma \ref{v_3(theta)-v_1(theta) bound lemma},
\begin{align}
|J_1^*(m)-J^*(m)|& \leqslant \int_{-N^{\delta-3}}^{N^{\delta-3}}\bigg \lvert v_1(\theta)^s-\left(\left(\frac{1}{3}\right)^{1/3}v(\theta)\right)^s\bigg \rvert \dd\theta\notag\\
&\leqslant 7N^{\delta}\int_{-N^{\delta-3}}^{N^{\delta-3}}\sum_{j=1}^s \lvert v_1(\theta) \rvert^{s-j}\cdot \bigg \lvert \left(\frac{1}{3}\right)^{1/3}v(\theta)\bigg \rvert^{j-1} \dd\theta\notag\\
&\leqslant 14sN^\delta N_0^{(s-1)/3}N^{\delta-3} \leqslant 21s\left(\frac{1}{3}\right)^{s/3}N^{s-4+2\delta}\label{J_1*(m)-J*(m)}.
\end{align}

Extending the integral $J^{*}_1(m)$ to an integral over a unit interval, we define
\begin{align}
J_1(m) := \int_{-1/2}^{1/2}v_1(\theta)^se(-\theta m)\dd\theta.\label{J_1(m)}
\end{align}

We approximate $J^{*}(m)$ with $J_1(m)$. 
By Lemma 7.2 in \cite{ours}, 
\[
|v_{1}(\theta)| \leqslant \min \{ 2m^{1/3}, 2 |\theta|^{-1/3}\}. \]

Moreover, $N^{\delta-3}< \frac{1}{2}$, so we have
\begin{align}
|J_1(m)-J_1^*(m)| &\leqslant \bigg|\int_{-1/2}^{-N^{\delta-3}} v_1(\theta)^se(-\theta m)\dd\theta\bigg|+\bigg|\int_{N^{\delta-3}}^{1/2} v_1(\theta)^se(-\theta m)\dd\theta\bigg|\notag\\
    &\leqslant 2\int_{N^{\delta-3}}^{1/2}\bigg(\min \bigg \{2m^{1/3}, 2 |\theta|^{-1/3} \bigg \}\bigg)^s \dd\theta\notag\\
    &\leqslant 2^{s+1}\int_{N^{\delta-3}}^{1/2}\theta^{-s/3}\dd\theta \leqslant 2^{s+1}\frac{3}{s-3}N^{\frac{3\delta-s\delta+3s-9}{3}}.\label{J_1(m)-J_^*(m)}
\end{align}

Combining \eqref{J_1*(m)-J*(m)} with \eqref{J_1(m)-J_^*(m)}, we obtain
\begin{align}
|J_1(m)-J^*(m)| \leqslant 21s\left(\frac{1}{3}\right)^{s/3}N^{s-4+2\delta}+\frac{2^s \cdot 6}{s-3}N^{\frac{3\delta-s\delta+3s-9}{3}}\label{J_1(m)-J^*(m)}.
\end{align} 
\begin{lem}\label{J_1(m) approximation bound}
Let $s\geqslant 2$. Then 
\begin{align*}
\bigg|J_1(m,s)-\Gamma\bigg(\frac{4}{3}\bigg)^s\Gamma\bigg(\frac{s}{3}\bigg)^{-1}m^{s/3-1}\bigg|\leqslant 10^{s-2}m^{(s-1)/3-1}.
\end{align*}

\end{lem}
\begin{proof}
By orthogonality, we may rewrite 
\[
J_1(m,s) =3^{-s}\mathop{\sum_{n_1=1}^{N_0}\cdots\sum_{n_s=1}^{N_0}}_{n_1+\cdots+n_s=m} (n_1n_2\cdots n_s)^{-2/3}.
\]
\cite[Lemma 7.3]{ours} with $\alpha=\beta=\frac{1}{3}$ shows that Lemma \ref{J_1(m) approximation bound} holds for $s=2$. The genral case follows then by induction. See \cite[Lemma~7.4]{ours} for details.
\end{proof}

\section{Asymptotic Results for Representations as Sums of Square Pyramidal Numbers}\label{sec: Proof of Main Theorem} 
In this section, we will prove Lemma \ref{main Theorem for formula} and establish more general versions of this result. Throughout the calculations, we eventually switch estimates involving $N$ to those involving $m$ using \eqref{Defining N}. Same as before, we use the bound $(3m)^{1/3}<N<2m^{1/3}$.
\begin{lem}\label{main Theorem for formula}
For $m,s \in \mathbb{N}$, let $\mathcal{C}_s(m)$ denote the number of ways of representations of $m$ as the sum of $s$ square pyramidal numbers. Then for any $s\geqslant 9$, $0<\delta<\frac{1}{5}$, and $m>\frac{1}{3}(e^{e^{45}})^{3/\delta}$,
\begin{align}
\bigg|\mathcal{C}_s(m)& -3^{\frac{s}{3}}\mathfrak{S}(m)\Gamma\bigg(\frac{4}{3}\bigg)^{s}\Gamma\bigg(\frac{s}{3}\bigg)^{-1}m^{\frac{s}{3}-1}\bigg| \notag \\
     &\leqslant 50s\left(\frac{m}{8}\right)^{\frac{5\delta+s-4}{3}}+ \frac{(6\cdot 2^s+2\cdot {20}^{s/3}(s-3))\cdot 6^s\cdot 3^{\frac{21s-(5s-42)\delta}{63}}}{(s-3)(\frac{5s}{21}-2)} \cdot m^{\frac{(42-5s)\delta+21s-63}{63}}\notag\\
&\quad +21s\cdot 2^{s}e^{e^{46}} m^{\frac{s-4+2\delta}{3}}+24^{\frac{s}{3}}\cdot 2^{\frac{3\delta-s\delta-9}{3}}\cdot \frac{6e^{e^{46}}2^s}{s-3}\cdot m^{\frac{3\delta-s\delta+3s-9}{9}}\notag\\
&\quad + 3^{\frac{s}{3}} \left( e^{e^{46}} \cdot 10^{s-2} m^{\frac{s-4}{3}}\right) +152 \cdot 13^{s-8} \cdot m^{\frac{s}{3}-1-\frac{\delta(s-8)}{12}+\frac{0.53305(s-8)+6.3966}{1.2\log \log m}+\frac{(s-8)\log\log m}{4\log m}} \label{Icosahedral 1st Theorem Eq},
\end{align}
where $\mathfrak{S}(m)$ satisfies \eqref{Sigma Bound}.
\end{lem} 
\begin{proof} 
For $s>3$ and $\lvert\theta\rvert\leqslant\frac{1}{2}$, we have the bound $\left\lvert \left(\frac{1}{3}\right)^{1/3}v(\theta)\right\rvert\leqslant\min\left\{\left(\frac{1}{3}\right)^{1/3}N,2\lvert\theta\rvert^{-1/3}\right\}.$
Therefore, 
\begin{align}
|J^*(m)|&= \left|\int_{-N^{\delta-3}}^{N^{\delta-3}}\left(\frac{1}{3}\right)^{s/3}v^s(\theta)e(-\theta m) \dd\theta\right| \leqslant \left(\frac{2^{s+1} \cdot 3}{s-3}+2\cdot 20^{s/3}\right)m^{\frac{s}{3}-1}\label{|J^*(m)|}.
\end{align}

Combining Lemma \ref{Singular Series Extension}, \eqref{J_1(m)-J^*(m)} and \eqref{|J^*(m)|}, we deduce that
\begin{align}
\bigg|\mathfrak{S}_{s}^*(m)- 3^{\frac{s}{3}}\mathfrak{S}(m)J_1(m)\bigg| 
    &\leqslant \bigg(\frac{2}{5}\bigg)^{\frac{s}{3}}\bigg(|J^*(m)|\cdot |\mathfrak{S}(m)-\mathfrak{S}(m, N^\delta)|+|\mathfrak{S}(m)||J_1(m)-J^*(m)|\bigg)\notag\\
    &\leqslant \frac{(6\cdot 2^s+2\cdot 20^{\frac{s}{3}}(s-3))\cdot 6^s\cdot 3^{\frac{21s-(5s-42)\delta}{63}}}{(s-3)(\frac{5s}{21}-2)}\cdot m^{\frac{(42-5s)\delta+21s-63}{63}} \notag\\
    &\quad+21s\cdot 2^{s}e^{e^{46}} m^{\frac{s-4+2\delta}{3}}+24^{\frac{s}{3}}\cdot 2^{\frac{3\delta-s\delta-9}{3}}\cdot \frac{6e^{e^{46}}2^s}{s-3}\cdot m^{\frac{3\delta-s\delta+3s-9}{9}}.\label{R*(m)-(2/5)^{s/3}S(m)J_1(m)}
\end{align}

Thus, in \eqref{Approximating Major Arc Integral}, one may replace $J^{*}(m)$ by $J_1(m)$ while allowing a small error.  Combining Lemma \ref{Approx 4} \mbox{and \eqref{R*(m)-(2/5)^{s/3}S(m)J_1(m)},} 
\begin{align}
\bigg|\int_{\mathfrak{M}}&f(\alpha)^s e(-\alpha m) \dd\alpha-3^{\frac{s}{3}}\mathfrak{S}(m)J_1(m)\bigg|\notag\\
&\leqslant
50s\left(\frac{m}{8}\right)^{\frac{5\delta+s-4}{3}}+ \frac{(6\cdot 2^s+2\cdot {20}^{s/3}(s-3))\cdot 6^s\cdot 3^{\frac{21s-(5s-42)\delta}{63}}}{(s-3)(\frac{5s}{21}-2)} \cdot m^{\frac{(42-5s)\delta+21s-63}{63}}\notag\\
&\quad +21s\cdot 2^{s}e^{e^{46}} m^{\frac{s-4+2\delta}{3}}+24^{\frac{s}{3}}\cdot 2^{\frac{3\delta-s\delta-9}{3}}\cdot \frac{6e^{e^{46}}2^s}{s-3}\cdot m^{\frac{3\delta-s\delta+3s-9}{9}}.
\label{f(alpha)-S(m)J_1(m) bound}
\end{align}

Using Lemma \ref{J_1(m) approximation bound} and \eqref{f(alpha)-S(m)J_1(m) bound}, with $s\geqslant 9$ and $m>\frac{1}{3}(e^{e^{45}})^{3/\delta}$ , 
\begin{align}
\bigg|\int_{\mathfrak{M}}&f(\alpha)^s e(-\alpha m) \dd\alpha-3^{\frac{s}{3}}\Gamma\bigg(\frac{4}{3}\bigg)^{s}\Gamma\bigg(\frac{s}{3}\bigg)^{-1}\mathfrak{S}(m) m^{\frac{s}{3}-1}\bigg|\notag\\
 &\leqslant
50s\left(\frac{m}{8}\right)^{\frac{5\delta+s-4}{3}}+ \frac{(6\cdot 2^s+2\cdot {20}^{s/3}(s-3))\cdot 6^s\cdot 3^{\frac{21s-(5s-42)\delta}{63}}}{(s-3)(\frac{5s}{21}-2)} \cdot m^{\frac{(42-5s)\delta+21s-63}{63}}\notag\\
&\quad +21s\cdot 2^{s}e^{e^{46}} m^{\frac{s-4+2\delta}{3}}+24^{\frac{s}{3}}\cdot 2^{\frac{3\delta-s\delta-9}{3}}\cdot \frac{6e^{e^{46}}2^s}{s-3}\cdot m^{\frac{3\delta-s\delta+3s-9}{9}} + 3^{\frac{s}{3}} \left( e^{e^{46}} \cdot 10^{s-2} m^{\frac{s-4}{3}}\right). \label{4.26 bounds}
\end{align}

By Lemma \ref{Minor Arcs : thm} and \eqref{4.26 bounds}, the claim is proved.
\end{proof}
\begin{rem}\label{main theorem asymp} We now optimize our choice of $\delta$ in terms of $s$ to minimize the exponents of $m$ in the bound provided in the statement of Lemma \ref{main Theorem for formula}. 
Note that these six exponents are exactly the same as in \cite[Theorem 8.2]{ours}. Therefore, by \cite[Theorem 8.2]{ours},  one may choose $\delta = \frac{21}{63+5s}$. In fact, such an observation is not occusional. See Remark \ref{rem: last remark on general polynomials} for details.
\end{rem}
If we let $s=9$, and $V(q,a)$ and $S_n$ be given by \eqref{Definition of V(q,a)} and \eqref{defn: square pyramidal number} respectively, we may define
\begin{align}\label{Arithmetic Factor for s=9}
\mathfrak{S}_{9,S}(m) : = \sum_{q=1}^{\infty} \sum\limits_{\substack{a=1 \\ (a,q)=1}}^q\bigg(\frac{V(q,a)}{6q}\bigg)^9 e\bigg(-\frac{am}{q}\bigg), \nonumber
\end{align}
which yields Lemma \ref{main Theorem for formula}.  To ensure that in Theorem \ref{Theorem: Representation}, the error term is smaller than the main term, we only need $m>e^{e^{94}}$. Lemmas \ref{main Theorem for formula} and \ref{main theorem asymp} hold for a much lower threshold for $m$.\par

Now, with Lemma \ref{main Theorem for formula}, we immediately obtain the following cannonball polygon result. 
\begin{lem}\label{lem: representations of polygons}
Let $\mathfrak{C}_s(\mathcal{Z})$ denote the number of classes of cannonball polygons with multiplicity $s$ and with largest side of length $\mathcal{Z}$. Then for any $s\geqslant 9$, $0<\delta<\frac{1}{5}$, and $\mathcal{Z}^2>\frac{1}{3}(e^{e^{45}})^{3/\delta}$,
\begin{align}
\bigg|\mathfrak{C}_s(\mathcal{Z})& -3^{\frac{s}{3}}\mathfrak{S}(\mathcal{Z}^2)\Gamma\bigg(\frac{4}{3}\bigg)^{s}\Gamma\bigg(\frac{s}{3}\bigg)^{-1}\mathcal{Z}^{\frac{2s}{3}-2}\bigg| \notag \\
     &\leqslant 50s\left(\frac{\mathcal{Z}^2}{8}\right)^{\frac{5\delta+s-4}{3}}+ \frac{(6\cdot 2^s+2\cdot {20}^{s/3}(s-3))\cdot 6^s\cdot 3^{\frac{21s-(5s-42)\delta}{63}}}{(s-3)(\frac{5s}{21}-2)} \cdot {\mathcal{Z}}^{\frac{(84-10s)\delta+42s-126}{63}}\notag\\
&\quad +21s\cdot 2^{s}e^{e^{46}} {\mathcal{Z}}^{\frac{2s-8+4\delta}{3}}+24^{\frac{s}{3}}\cdot 2^{\frac{3\delta-s\delta-9}{3}}\cdot \frac{6e^{e^{46}}2^s}{s-3}\cdot {\mathcal{Z}}^{\frac{6\delta-2s\delta+6s-18}{9}}\notag\\
&\quad + 3^{\frac{s}{3}} \left( e^{e^{46}} \cdot 10^{s-2} {\mathcal{Z}}^{\frac{2s-8}{3}}\right) +152 \cdot 13^{s-8} \cdot {\mathcal{Z}}^{\frac{2s}{3}-2-\frac{\delta(s-8)}{6}+\frac{1.0661(s-8)+12.7932}{1.2\log \log m}+\frac{(s-8)\log\log m}{2\log m}},
\end{align}
where $\mathfrak{S}(\mathcal{Z}^2)$ satisfies \eqref{Sigma Bound}. 
\end{lem}
\begin{rem}\label{rem: last remark on general polynomials}
    In fact, such an asymptotic formula is believed to hold for any degree-three polynomial $f$ such \mbox{that $f(0)=0$ and $f(1)=1$,} with only potentially a different lower bound on the arithmetic factor $\mathfrak{S}(\mathcal{Z}^2)$ depending on $f$. One can see an established conclusion for degree-four polynomials in \cite[Theorem 1.3 and its generalized version]{DTZ25}. 
\end{rem}
\section*{Acknowledgements}
The authors are also grateful to the referee for many valuable suggestions that have greatly increased the clarity and value of the manuscript.

\end{document}